\documentclass[11pt]{amsart}
\usepackage[utf8]{inputenc}
\usepackage{setspace}
\usepackage{geometry}
\usepackage{xcolor, amsmath, amssymb, latexsym, bbm, mathtools}
\usepackage[mathscr]{euscript}
\usepackage{enumitem}
\usepackage{tikz}
\usepackage{wrapfig}
\usepackage{graphicx}
\usepackage{subcaption}
\usepackage{float}
\usepackage{comment}
\usepackage{thmtools}
\usetikzlibrary{calc}
\usetikzlibrary{arrows.meta}
\usetikzlibrary{arrows.meta}
\usetikzlibrary{decorations.markings}

\usepackage[colorlinks=true, citecolor=blue, linkcolor=blue, urlcolor=blue]{hyperref}

\calclayout
\numberwithin{equation}{section}
\numberwithin{figure}{section}
\allowdisplaybreaks

\theoremstyle{plain}
\newtheorem{theorem}{Theorem}[section]
\newtheorem*{theorem*}{Theorem}

\newtheorem{lemma}[theorem]{Lemma}
\newtheorem{proposition}[theorem]{Proposition}

\theoremstyle{definition}
\newtheorem{definition}[theorem]{Definition}
\newtheorem{question}[theorem]{Question}

\theoremstyle{remark}
\newtheorem{remark}[theorem]{Remark}
\newtheorem{note}[theorem]{Note}

\title{ Systolic Increasing Deformations To Maximal Translation Surfaces }
\author{Achintya Dey}

\address{
Department of Mathematics and Statistics\\ 
Indian Institute of Technology  \\ 
Kanpur, Uttar Pradesh-208016\\
India}
\email{achintyadey77459@gmail.com}
\author{Bidyut Sanki}
\address{
Department of Mathematics and Statistics\\ 
Indian Institute of Technology  \\ 
Kanpur, Uttar Pradesh-208016\\
India}
\email{bidyut@iitk.ac.in}

\date{\today}

\begin{document}

\begin{abstract}
A unit-area translation surface is called \emph{maximal} if it maximizes the length of its shortest saddle connection among all surfaces in the same stratum. We investigate whether a non-maximal translation surface can be continuously deformed into a maximal one while the systole increases strictly monotonically. Although such deformations do not exist in general due to the presence of local, non-global maxima of the systole function, we explicitly construct them for the family of so-called $2k$-gonal translation surfaces for any fixed $k \ge 2$. Furthermore, we generalize this construction to translation surfaces composed of mixed polygons. Specifically, we prove that any translation surface formed by gluing a combination of coherently aligned regular $2k$-gons with potentially varying values of $k$, provided every polygon side length equals the surface's systole, also admits a continuous, strictly systole-increasing deformation terminating at a maximal surface.
\end{abstract}

\maketitle {}
\section{Introduction}

The study of translation surfaces has become a central theme in modern geometry, with deep connections to flat geometry, Teichmüller theory, and dynamical systems. A \emph{translation surface} is obtained from a finite union of Euclidean polygons by identifying parallel sides of opposite orientation via translations. The resulting surface is flat except at a finite set of conical \emph{singular points}, each having total angle $2\pi(k+1)$, for some $k\in\mathbb{N}$. A \emph{saddle connection} on a translation surface is a geodesic (straight line) segment joining two (possibly identical) singular points and containing no singular point in its interior. Among all saddle connections, those of minimal length are called the \emph{shortest saddle connections or systolic connections}, and their common length is referred to as the \emph{systole} of the surface. For integers $k_1,\dots,k_n \ge 1$, the \emph{stratum} $\mathcal{H}(k_1,\dots,k_n)$ consists of all translation surfaces with fixed genus $g$ and exactly $n$ singularities having cone angles $2\pi(k_i+1)$. Throughout this article, we restrict to translation surfaces of unit area and \(\mathcal{H}(k_1,\dots,k_n)\) denotes the stratum of unit-area translation surfaces.
Within a fixed stratum, surfaces that maximize the systole are called \emph{maximal surfaces}. 

In the hyperbolic setting, the systole function $\mathrm{sys}:\mathcal{T}_g\to\mathbb{R}$
is defined on the Teichm\"uller space $\mathcal{T}_g$ \cite{farb2011primer} of closed hyperbolic surfaces of genus $g\geq 2$, where for a hyperbolic surface $X\in\mathcal{T}_g$, the value $\mathrm{sys}(X)$ is the length of the shortest non-contractible simple closed geodesic on $X$. Hyperbolic surfaces with a filling set of systolic geodesics are closely related to Thurston’s spine problem. A spine of the Teichmüller space $\mathcal{T}_g$ is a mapping class group-equivariant deformation retract of minimal dimension. Thurston claimed that the set $\chi_g$ consisting of hyperbolic surfaces whose systolic geodesics fill the surface, known as the Thurston set, is a potential candidate for a spine. He sketched a proof in \cite{thurstonspine}, although the proof is incomplete.

A hyperbolic surface is called a \emph{global} (respectively, \emph{local}) \emph{maximal surface} if it realizes a global (respectively, local) maximum of the systole function. Despite extensive study, the global maxima of the systole function are known only in genus $2$. Jenni~\cite{Jenni} proved that the unique maximal surface in $\mathcal{T}_2$ is the Bolza surface and showed that
$\cosh(\mathrm{sys}(X)/2)\leq (1+\sqrt{2})$ for every $X\in\mathcal{T}_2$, with equality if and only if $X$ is the Bolza surface. Schmutz~\cite{SchmutzGlobalMaximal} further proved that every maximal surface of genus $g$ contains at least $6g-5$ systolic geodesics (see Theorem~2.8 of \cite{SchmutzGlobalMaximal}). More recently, Dey--Saha--Sanki~\cite{DeySahaSankiToAppear} constructed global maximal surfaces of genus $g$ with $2g-2$ punctures, where $g=kn+1$ for integers $k\geq 2$ and $n\geq 4$. Their construction provides explicit examples of global maxima of the systole function in infinitely many topological types.

The analogous problem for translation surfaces has attracted considerable attention in recent years. A systematic study of maximal translation surfaces was carried out by Boissy--Geninska~\cite{Bo21}. The authors \cite{Bo21} established that, in every area-one stratum $\mathcal{H}(k_1,\dots,k_n)$ of translation surfaces of genus $g$, the global maxima of the systole function are realized precisely by translation surfaces obtained by glueing equilateral triangles whose side length is $\left(\frac{\sqrt{3}}{2}(2g-2+n)\right)^{-\frac{1}{2}}.$ Moreover, these surfaces have exactly $\sum_{i=1}^{n}3(k_i+1)$ shortest saddle connections, which is the maximum possible number among all surfaces in the stratum $\mathcal{H}(k_1,k_2,\dots,k_n)$. The study of systolic structures on translation surfaces has also been pursued from several other perspectives. Judge-Parlier~\cite{Pa19} determined the maximum number of systoles, in the classical sense of shortest non-contractible simple closed geodesics, on genus-two translation surfaces. Columbus-Herrlich-M\"utzel-Schmith\"usen~\cite{MR4790974} investigated maximal translation surfaces from the viewpoint of the homological systolic ratio. Their notion of maximality differs from the one adopted in this article, as they call a translation surface maximal if it attains the supremal homological systolic ratio. Recently, Dey--Sanki~\cite{MR4985436} studied systolic embeddings of graphs on translation surfaces and determined the extremal genera of translation surfaces admitting such embeddings. Thurston's \cite{thurstonspine} deformation argument provides a useful motivation for our work. In the setting of hyperbolic surfaces, Thurston showed that a surface outside the Thurston set $\chi_g$ can be continuously deformed so as to increase its systole, with the deformation eventually leading toward the set $\chi_g$. More precisely, starting with a hyperbolic surface that is not in $\chi_g$, one can construct a deformation, so-called strip deformation which produces a surface of the same genus but with strictly larger systolic length. This naturally raises an analogous question for translation surfaces. Given a translation surface in a fixed stratum which is not maximal with respect to its systolic length, can one continuously deform it within the same stratum while monotonically increasing its systole until a maximal translation surface is reached? In other words, can the deformation principle underlying Thurston's construction be extended to the setting of translation surfaces? More precisely, we study the question below.

\begin{question} \label{main question}
Let X $\in \mathcal{H}(k_1,\dots,k_n)$ be a translation surface which is not a global maximum of the systole function on $\mathcal{H}(k_1,\dots,k_n)$. Does there exist a continuous function $\gamma: [0, 1]\to \mathcal{H}(k_1, \dots, k_n)$ such that 
\begin{enumerate}
    \item $\gamma(0)$= X and $\gamma(1)$ is a maximal surface and
    \item  $\mathrm{sys}(\gamma(t^\prime))<\mathrm{sys}(\gamma(t))$, for all $t^\prime < t\in[0,1]$?
\end{enumerate}
\end{question}

In general, the answer to Question~\ref{main question} is negative. Indeed, if $X$ is a local maximum (not global) of the systole function, then there is no continuous deformation of $X$ within the stratum along which the systole increases monotonically. Consequently, such a surface cannot be deformed to a maximal translation surface while strictly increasing its systole throughout the deformation. Examples of local but not global maxima of the systole function can be found in~\cite{Bo21} (see Theorem~4.1, ~\cite{Bo21}).

The natural question is whether there exist classes of translation surfaces in a fixed stratum for which Question~\ref{main question} has a positive answer. In this article, we give an affirmative answer to this question for each class $\mathcal{P}_k$, $k\ge 2$, of
\emph{$2k$-gonal translation surfaces} in a fixed stratum
$\mathcal{H}(k_1,\dots,k_n)$. Here, $\mathcal{P}_k$ denotes
the class of translation surfaces obtained by identifying pairs of parallel sides of finitely many regular Euclidean $2k$-gons of equal side length, with opposite orientations, via translations.

In particular, to address Question~\ref{main question}, we obtain the following main result:

\begin{restatable}{theorem}{SquareTiledTheorem}
\label{main result 1}
Let $X_{\mathrm{reg}}\in
\displaystyle\bigcup_{k\ge 2}\mathcal{P}_k$, in the unit-area stratum
$\mathcal{H}(k_1,\dots,k_n)$.
Then there exists a continuous map $\gamma:[0,1]\longrightarrow
\mathcal{H}_1(k_1,\dots,k_n),$
such that:
\begin{enumerate}
    \item $\gamma(0)=X_{\mathrm{reg}}$ and $\gamma(1)$ is a
    maximal surface in $\mathcal{H}_1(k_1,\dots,k_n)$;
    \item $\operatorname{sys}(\gamma(t'))
    <
    \operatorname{sys}(\gamma(t)), 
    \text{ for all }\ 0\le t'<t\le 1.$
\end{enumerate}
\end{restatable}

The idea of the proof of this theorem is the following:

For $k=2$, corresponding to square-tiled surfaces, we explicitly construct a systole-increasing deformation to a maximal surface. For $k\ge 3$, we first construct a systole-increasing deformation from a $2k$-gonal surface to a square-tiled surface, and then apply the deformation constructed in the case $k=2$ to reach a maximal surface.

Furthermore, using this result, we find another family of translation surfaces in a fixed stratum for which one can construct a systole-increasing deformation analogous to those for the families $\mathcal{P}_k$, $k\ge2$ (see Theorem~\ref{thm:polygonal_deformation}).

\section{Background}
In this section, we recall the definition of translation surfaces and related fundamental notions, including singular points, saddle connections, and stratum. We then introduce $2k$-gonal surfaces for $k\ge 2$ and define maximal translation surfaces, incorporating the notion of systole. We conclude the section by recalling a result concerning the sharp bound on systolic lengths of translation surfaces to characterize the maximal surfaces in a fixed stratum.

\begin{definition}\label{Translation Surface} (Masur \cite{Masur06})
A \emph{translation surface} is a finite union of Euclidean polygons $\{ \triangle_1, \triangle_2, \dots,\triangle_n \}$, equipped with a side pairing,
 such that:
\begin{enumerate}
    \item The boundary of every polygon is oriented so that the interior of the polygon lies to the left.
    \item For every $1\leq j\leq n$ and every oriented side $s_j$ of $\triangle_j$, there exists $k\in\{1,2,\cdots,n\}$ and an oriented side $s_k$ of $\triangle_k$ such that $s_j$ and $s_k$ are parallel and of the same length. They are glued together in the opposite orientation by a parallel translation.
\end{enumerate}
\end{definition}

A translation surface is obtained by identifying pairs of parallel sides of Euclidean polygons via translations. The vertices of these polygons are thereby identified to form cone points or singular points, and the total angle around each singular point is an integer multiple of $2\pi$ (see page $5$, \cite{Mas22}).
\begin{definition}
Let $X$ be a translation surface. A point of $X$ corresponding to an equivalence class of vertices of the defining polygons is called
\begin{enumerate}
    \item a \emph{singular point} if its cone angle is greater than $2\pi$;
    \item a \emph{marked point} if its cone angle is equal to $2\pi$.
\end{enumerate}
\end{definition}

\begin{definition}
A \emph{saddle connection} on a translation surface is a geodesic or straight line segment joining two singular points (possibly the same) whose interior contains no singular points.
\end{definition}

A saddle connection of minimal length is called a \emph{systolic connection (or shortest saddle connection)}.

\begin{note}
If a translation surface has no singular points, then marked points replace singular points in the above definition.
\end{note}

\begin{definition}[Stratum]
Let \(k_1,\ldots,k_n\) be positive integers. The \emph{stratum} \(\mathcal{H}(k_1,\ldots,k_n)\) is the set of all translation surfaces of a fixed genus
having exactly \(n\) singular points with cone angles
\(2\pi(k_i+1)\), \(i=1,\ldots,n\).
\end{definition}

If a translation surface \(X\in\mathcal{H}(k_1,\ldots,k_n)\) has genus \(g\) then it follows from the Gauss-Bonnet theorem (see Section $3.1$ of \cite{Mas22}), that
\[
\sum_{i=1}^n k_i=2g-2.
\]

\begin{remark}
Whenever marked points are present on a translation surface, we include them in the stratum notation by appending zero entries, with each zero corresponding to a marked point. Thus, a stratum of the form $\mathcal{H}(k_1,\ldots,k_n)$, where $k_i\geq 1$ for every $i$, contains no marked points.
\end{remark}

\begin{definition}[$2k$-gonal translation surface]
A \emph{$2k$-gonal translation surface}, for $k\ge 2$, is a translation surface constructed by taking finitely many regular Euclidean $2k$-polygons of the same side length, and identifying pairs of parallel sides with opposite orientations via translations. 
\end{definition}

For example, Figures~\ref{Fig: Hexagonal surface Example} and \ref{Fig: Octagonal surface Example} depict the corresponding surfaces for $k=3$ and $k=4$, respectively.

\begin{figure}[htbp]
    \centering
    \begin{tikzpicture}[
        scale=1.4,
        vred/.style={circle, fill=red, minimum size=3pt, inner sep=0pt},
        vblue/.style={circle, fill=blue, minimum size=3pt, inner sep=0pt},
        elabel/.style={text=purple}
    ]

    \begin{scope}[shift={(0,0)}]
        \coordinate (V1) at (120:1);
        \coordinate (V2) at (60:1);
        \coordinate (V3) at (0:1);
        \coordinate (V4) at (300:1);
        \coordinate (V5) at (240:1);
        \coordinate (V6) at (180:1);

        \draw (V1) -- (V2) node[midway, above, elabel] {\tiny$e_4$};
        \draw (V2) -- (V3) node[midway, above right, elabel] {\tiny$e_9$};
        \draw (V3) -- (V4) node[midway, below right, elabel] {\tiny$e_5$};
        \draw (V4) -- (V5) node[midway, below, elabel] {\tiny$e_1$};
        \draw (V5) -- (V6) node[midway, below left, elabel] {\tiny$e_{12}$};
        \draw (V6) -- (V1) node[midway, above left, elabel] {\tiny$e_8$};

        \node[vred] at (V1) {}; \node[vblue] at (V2) {}; \node[vred] at (V3) {};
        \node[vblue] at (V4) {}; \node[vred] at (V5) {}; \node[vblue] at (V6) {};
    \end{scope}

    \begin{scope}[shift={(2.5,0)}]
        \coordinate (V1) at (120:1);
        \coordinate (V2) at (60:1);
        \coordinate (V3) at (0:1);
        \coordinate (V4) at (300:1);
        \coordinate (V5) at (240:1);
        \coordinate (V6) at (180:1);

        \draw (V1) -- (V2) node[midway, above, elabel] {\tiny$e_1$};
        \draw (V2) -- (V3) node[midway, above right, elabel] {\tiny$e_{12}$};
        \draw (V3) -- (V4) node[midway, below right, elabel] {\tiny$e_7$};
        \draw (V4) -- (V5) node[midway, below, elabel] {\tiny$e_2$};
        \draw (V5) -- (V6) node[midway, below left, elabel] {\tiny$e_{11}$};
        \draw (V6) -- (V1) node[midway, above left, elabel] {\tiny$e_6$};

        \node[vred] at (V1) {}; \node[vblue] at (V2) {}; \node[vred] at (V3) {};
        \node[vblue] at (V4) {}; \node[vred] at (V5) {}; \node[vblue] at (V6) {};
    \end{scope}

    \begin{scope}[shift={(5.0,0)}]
        \coordinate (V1) at (120:1);
        \coordinate (V2) at (60:1);
        \coordinate (V3) at (0:1);
        \coordinate (V4) at (300:1);
        \coordinate (V5) at (240:1);
        \coordinate (V6) at (180:1);

        \draw (V1) -- (V2) node[midway, above, elabel] {\tiny$e_2$};
        \draw (V2) -- (V3) node[midway, above right, elabel] {\tiny$e_{11}$};
        \draw (V3) -- (V4) node[midway, below right, elabel] {\tiny$e_8$};
        \draw (V4) -- (V5) node[midway, below, elabel] {\tiny$e_3$};
        \draw (V5) -- (V6) node[midway, below left, elabel] {\tiny$e_{10}$};
        \draw (V6) -- (V1) node[midway, above left, elabel] {\tiny$e_7$};

        \node[vred] at (V1) {}; \node[vblue] at (V2) {}; \node[vred] at (V3) {};
        \node[vblue] at (V4) {}; \node[vred] at (V5) {}; \node[vblue] at (V6) {};
    \end{scope}

    \begin{scope}[shift={(7.5,0)}]
        \coordinate (V1) at (120:1);
        \coordinate (V2) at (60:1);
        \coordinate (V3) at (0:1);
        \coordinate (V4) at (300:1);
        \coordinate (V5) at (240:1);
        \coordinate (V6) at (180:1);

        \draw (V1) -- (V2) node[midway, above, elabel] {\tiny$e_3$};
        \draw (V2) -- (V3) node[midway, above right, elabel] {\tiny$e_{10}$};
        \draw (V3) -- (V4) node[midway, below right, elabel] {\tiny$e_6$};
        \draw (V4) -- (V5) node[midway, below, elabel] {\tiny$e_4$};
        \draw (V5) -- (V6) node[midway, below left, elabel] {\tiny$e_9$};
        \draw (V6) -- (V1) node[midway, above left, elabel] {\tiny$e_5$};

        \node[vred] at (V1) {}; \node[vblue] at (V2) {}; \node[vred] at (V3) {};
        \node[vblue] at (V4) {}; \node[vred] at (V5) {}; \node[vblue] at (V6) {};
    \end{scope}

    \end{tikzpicture}
    \caption{A hexagonal surface in the stratum $\mathcal{H}(3,3).$}
    \label{Fig: Hexagonal surface Example}
\end{figure}

\begin{figure}[htbp]
    \centering
    \begin{tikzpicture}[
        scale=1.4,
        vred/.style={circle, fill=red, minimum size=3pt, inner sep=0pt},
        vblue/.style={circle, fill=cyan, minimum size=3pt, inner sep=0pt},
        vblack/.style={circle, fill=black, minimum size=3pt, inner sep=0pt},
        elabel/.style={text=purple}
    ]

    \begin{scope}[shift={(0,0)}]
        \coordinate (V1) at (112.5:1.2);
        \coordinate (V2) at (67.5:1.2);
        \coordinate (V3) at (22.5:1.2);
        \coordinate (V4) at (337.5:1.2);
        \coordinate (V5) at (292.5:1.2);
        \coordinate (V6) at (247.5:1.2);
        \coordinate (V7) at (202.5:1.2);
        \coordinate (V8) at (157.5:1.2);

        \draw (V1) -- (V2) node[midway, above, elabel] {\tiny$f_1$};
        \draw (V2) -- (V3) node[midway, above right, elabel] {\tiny$f_9$};
        \draw (V3) -- (V4) node[midway, right, elabel] {\tiny$f_{12}$};
        \draw (V4) -- (V5) node[midway, below right, elabel] {\tiny$f_4$};
        \draw (V5) -- (V6) node[midway, below, elabel] {\tiny$f_6$};
        \draw (V6) -- (V7) node[midway, below left, elabel] {\tiny$f_7$};
        \draw (V7) -- (V8) node[midway, left, elabel] {\tiny$f_{10}$};
        \draw (V8) -- (V1) node[midway, above left, elabel] {\tiny$f_3$};

        \node[vblack] at (V1) {}; \node[vred] at (V2) {}; \node[vred] at (V3) {}; \node[vred] at (V4) {};
        \node[vblack] at (V5) {}; \node[vblue] at (V6) {}; \node[vblue] at (V7) {}; \node[vblue] at (V8) {};
    \end{scope}

    \begin{scope}[shift={(3.4,0)}]
        \coordinate (V1) at (112.5:1.2);
        \coordinate (V2) at (67.5:1.2);
        \coordinate (V3) at (22.5:1.2);
        \coordinate (V4) at (337.5:1.2);
        \coordinate (V5) at (292.5:1.2);
        \coordinate (V6) at (247.5:1.2);
        \coordinate (V7) at (202.5:1.2);
        \coordinate (V8) at (157.5:1.2);

        \draw (V1) -- (V2) node[midway, above, elabel] {\tiny$f_6$};
        \draw (V2) -- (V3) node[midway, above right, elabel] {\tiny$f_8$};
        \draw (V3) -- (V4) node[midway, right, elabel] {\tiny$f_{11}$};
        \draw (V4) -- (V5) node[midway, below right, elabel] {\tiny$f_3$};
        \draw (V5) -- (V6) node[midway, below, elabel] {\tiny$f_5$};
        \draw (V6) -- (V7) node[midway, below left, elabel] {\tiny$f_9$};
        \draw (V7) -- (V8) node[midway, left, elabel] {\tiny$f_{12}$};
        \draw (V8) -- (V1) node[midway, above left, elabel] {\tiny$f_2$};

        \node[vblue] at (V1) {}; \node[vblack] at (V2) {}; \node[vblack] at (V3) {}; \node[vblack] at (V4) {};
        \node[vblue] at (V5) {}; \node[vred] at (V6) {}; \node[vred] at (V7) {}; \node[vred] at (V8) {};
    \end{scope}

    \begin{scope}[shift={(6.8,0)}]
        \coordinate (V1) at (112.5:1.2);
        \coordinate (V2) at (67.5:1.2);
        \coordinate (V3) at (22.5:1.2);
        \coordinate (V4) at (337.5:1.2);
        \coordinate (V5) at (292.5:1.2);
        \coordinate (V6) at (247.5:1.2);
        \coordinate (V7) at (202.5:1.2);
        \coordinate (V8) at (157.5:1.2);

        \draw (V1) -- (V2) node[midway, above, elabel] {\tiny$f_5$};
        \draw (V2) -- (V3) node[midway, above right, elabel] {\tiny$f_7$};
        \draw (V3) -- (V4) node[midway, right, elabel] {\tiny$f_{10}$};
        \draw (V4) -- (V5) node[midway, below right, elabel] {\tiny$f_2$};
        \draw (V5) -- (V6) node[midway, below, elabel] {\tiny$f_1$};
        \draw (V6) -- (V7) node[midway, below left, elabel] {\tiny$f_8$};
        \draw (V7) -- (V8) node[midway, left, elabel] {\tiny$f_{11}$};
        \draw (V8) -- (V1) node[midway, above left, elabel] {\tiny$f_4$};

        \node[vred] at (V1) {}; \node[vblue] at (V2) {}; \node[vblue] at (V3) {}; \node[vblue] at (V4) {};
        \node[vred] at (V5) {}; \node[vblack] at (V6) {}; \node[vblack] at (V7) {}; \node[vblack] at (V8) {};
    \end{scope}

    \end{tikzpicture}
    \caption{An octagonal surface in the stratum $\mathcal{H}(2,2,2).$}
    \label{Fig: Octagonal surface Example}
\end{figure}

Let $\mathcal{H}(k_1,\dots,k_n)$ be a stratum of unit-area translation surfaces. 
For a surface $X \in \mathcal{H}(k_1,\dots,k_n)$, the \emph{systole} of $X$ is defined as the length of a systolic connection of $X$ and is denoted by $\mathrm{sys}(X).$
\begin{definition}[Maximal translation surface]
We say that $X$ is a \emph{maximal translation surface} if $$\mathrm{sys}(X) \ge \mathrm{sys}(Y)\ \text{for all}\ Y \in \mathcal{H}(k_1,\dots,k_n).$$
\end{definition}

The characterization of such surfaces is completely described in \cite{Bo21} by the following result.

\begin{theorem}[\cite{Bo21}] \label{maximal surfaces}
Let X be a translation surface of genus $g \geq 1$ and area one, with $r > 0$ number of singular points (or marked points). Then
\[
\mathrm{sys}(X) \leq \left( \frac{2}{\sqrt{3}(2g - 2 + r)} \right)^{\frac{1}{2}}.
\]
Equality holds if and only if X is obtained by gluing equilateral triangles whose sides are saddle connections of length $\mathrm{sys}(X)$. 
Furthermore, such a surface exists in every connected component of every stratum.
\end{theorem}
\begin{figure}[htbp]
\begin{center}
\begin{tikzpicture}[scale=2]
\draw (0,0)--(1,0)--(2,0)--(3,0)--(3.5,0.866025403)--(2.5,0.866025403)--(0.5, 0.866025403)--(0,0);
\draw (0.5,0.866025403)--(1,0)--(1.5,0.866025403)--(2,0)--(2.5,0.866025403)--(3,0);
\draw [-{Stealth[blue]}]
(2.5,0)--(2.51,0);
\draw [-{Stealth[blue]}]
(1,0.866025403)--(1.1,0.866025403);
\draw (2.5,-0.15) node{a};
\draw (1,1.016025404) node{a};
\draw [-{Stealth[green]}]
(1.5,0)--(1.51,0);
\draw [-{Stealth[green]}]
(2,0.866025403)--(2.1,0.866025403);
\draw (1.5,-0.15) node{b};
\draw (2,1.016025404) node{b};
\draw [-{Stealth[red]}]
(0.5,0)--(0.51,0);
\draw [-{Stealth[red]}]
(3,0.866025403)--(3.1,0.866025403);
\draw (0.5,-0.15) node{c};
\draw (3,1.016025404) node{c};

\draw [-{Stealth}]
(0.25,0.433012701)--(0.26,0.45033321);
\draw [-{Stealth}]
(3.25,0.433012701)--(3.26,0.45033321);
\draw (0.10,0.433012701) node{d};
\draw (3.40,0.433012701) node{d};

\draw[red] (0,0) to (1,0);
\draw[red] (2.5,0.866025403) to (3.5,0.866025403);
\draw[green] (1.5, 0.866025403) to (2.5, 0.866025403);
\draw[green] (1,0) to (2,0);
\draw[blue] (0.5,0.866025403) to (1.5,0.866025403);
\draw[blue] (2,0) to (3,0);
\draw[yellow] (0,0) to (0.5,0.866025403);
\draw[yellow] (3,0) to (3.5, 0.866025403);
\draw[red] (0,0) node{\tiny$\bullet$};
\draw[red] (1,0) node{\tiny$\bullet$};
\draw[red] (2,0) node{\tiny$\bullet$};
\draw[red] (3,0) node{\tiny$\bullet$};
\draw[red] (3.5, 0.866025403) node{\tiny$\bullet$};
\draw[red] (2.5,0.866025403) node{\tiny$\bullet$};
\draw[red] (1.5,0.866025403) node{\tiny$\bullet$};
\draw[red] (0.5,0.866025403) node{\tiny$\bullet$};
\end{tikzpicture}
\end{center} 
\caption{Maximal surface of genus $2$ in $\mathcal{H}(2)$} \label{maxima}
\end{figure}
For example, a maximal surface of genus $2$ with one singular point is shown in Figure~\ref{maxima}.

\begin{note}
Throughout this paper, $\mathcal{H}(k_1,\ldots,k_n)$ denotes the stratum of unit-area translation surfaces.
\end{note}

\section{Deformations of Non-Maximal Translation Surfaces}

In this section, we construct explicit continuous deformations from
$2k$-gonal translation surfaces to maximal surfaces, along which the
systole increases strictly, for $k\ge 2$. These constructions establish
our main result, Theorem~\ref{main result 1}. We then use this result to
identify another class of surfaces admitting such systole-increasing
deformations; see Theorem~\ref{thm:polygonal_deformation}. For the
reader's convenience, we begin by restating Theorem~\ref{main result 1}.

\SquareTiledTheorem*

\begin{proof}
We consider two cases, depending on the class to which $X_{\mathrm{reg}}$ belongs.

\noindent\textbf{Case 1:} $X_{\mathrm{reg}}\in \mathcal{P}_2$, i.e., $X_{\mathrm{reg}}$ is a square-tiled surface.

Suppose the square-tiled surface $X_{\mathrm{reg}}$ is built from unit squares  $S_1,\dots,S_r$. Since $X_{\mathrm{reg}}$ belongs to the stratum $\mathcal{H}(k_1,k_2,\ldots,k_n)$, where $k_i\geq 1$ for every $i$, all vertex equivalence classes are singular points. Consequently, each side of a square determines a saddle connection of length $1$, and hence is a systolic connection. Thus, the systole of $X_{\mathrm{reg}}$ is equal to~$1$, and its area is $\mathrm{area}(X_{\mathrm{reg}})=r$.

We now introduce a one-parameter deformation of each square $S_i$ into a rhombus. For $\epsilon \in \left[0,\frac{\pi}{6}\right]$, let $S_\epsilon$ be the rhombus of side length $1$ with interior angles $ \frac{\pi}{2}-\epsilon, \frac{\pi}{2}+\epsilon, \frac{\pi}{2}-\epsilon~\text{and} \ \frac{\pi}{2}+\epsilon $. The concrete realization of $S_\epsilon$ is given by the vertices $A=(0,0),  B=(1,0),  C=(1+\sin\epsilon,\cos\epsilon) \ \text{and}\ D=(\sin\epsilon,\cos\epsilon).$ Equivalently, $S_\epsilon$ is the parallelogram spanned by the vectors $(1,0)$ and $(\sin\epsilon,\cos\epsilon)$ (see Figure~\ref{rhombus}). Thus, the area of \(S_\varepsilon\) is  $1\cdot1\cdot\sin\!\left(\tfrac{\pi}{2}-\epsilon\right)
= \cos\epsilon.$
Next we compute the length of the diagonal $d_\epsilon$ of $S_\epsilon$ opposite the angle $\frac{\pi}{2}-\epsilon$.  
By the law of cosines,
\[
d_\epsilon^{\,2}
=1^2+1^2-2\cdot 1\cdot 1\cdot \cos\!\left(\tfrac{\pi}{2}-\epsilon\right)
=2-2\sin\epsilon .
\]
Since $\sin\epsilon$ is strictly increasing in $[0,\tfrac{\pi}{2}]$, the function $(2-2\sin\epsilon)$ is strictly decreasing there. In particular, at $\epsilon=\frac{\pi}{6}$ we have $\sin(\frac{\pi}{6})=\frac12$, and hence $d_{\pi/6}=1$.  
Thus $S_{\pi/6}$ becomes a rhombus decomposed into two equilateral triangles. Now, replace each square $S_i$ in $X_{\mathrm{reg}}$ by a copy $S_\epsilon^i$ of $S_\epsilon$, glueing their sides according to the same pattern used in $X_{\mathrm{reg}}$.

\begin{figure}[htbp]
\centering
\begin{tikzpicture}[scale=3]
\draw (0,0) -- (1,0) -- (1.4,1) -- (0.4,1) -- cycle;

\draw[red, thick, dotted] (0.4,1) -- (1,0);
\draw (0.82,0.5) node {$d_\epsilon$};

\draw[blue, thick, dotted] (0.4,1) -- (0.4,0);
\draw (0.4,0.5) node[right] {\tiny$\cos(\epsilon)$};
\draw (0.4,0.05) -- (0.45,0.05) -- (0.45,0);

\draw (0.25, 0.1) node {\tiny$\frac{\pi}{2}-\epsilon$};
\draw (0.8, 0.12) node {\tiny$\frac{\pi}{2}+\epsilon$}; 
\draw (0.6, 0.9) node {\tiny$\frac{\pi}{2}+\epsilon$};
\draw (1.2, 0.9) node {\tiny$\frac{\pi}{2}-\epsilon$};
\end{tikzpicture}
\caption{$S_\epsilon$} 
\label{rhombus}
\end{figure}

\begin{figure}[htbp]
\centering
\begin{tikzpicture}[scale=1.1]

\begin{scope}
\draw (0,0)--(2,0)--(2,2)--(1,2)--(1,0);
\draw (0,0)--(0,1)--(3,1)--(3,2)--(2,2);

\draw[red] (3,2)--(2,2) node[midway, above] {$e_1$};
\draw[red] (2,1)--(3,1) node[midway, below] {$e_1$};

\draw[green!70!black] (1,2)--(2,2) node[midway, above] {$e_2$};
\draw[green!70!black] (1,0)--(2,0) node[midway, below] {$e_2$};

\draw[blue] (0,1)--(1,1) node[midway, above] {$e_3$};
\draw[blue] (0,0)--(1,0) node[midway, below] {$e_3$};

\draw[brown] (0,0)--(0,1) node[midway, left] {$e_4$};
\draw[brown] (2,0)--(2,1) node[midway, right] {$e_4$};

\draw[yellow!80!black] (1,1)--(1,2) node[midway, left] {$e_5$};
\draw[yellow!80!black] (3,1)--(3,2) node[midway, right] {$e_5$};

\draw (0,0) node{\tiny$\bullet$};
\draw (0,1) node{\tiny$\bullet$};
\draw (2,2) node{\tiny$\bullet$};
\draw (2,1) node{\tiny$\bullet$};
\draw (2,0) node{\tiny$\bullet$};

\draw[red] (1,2) node{\tiny$\bullet$};
\draw[red] (1,1) node{\tiny$\bullet$};
\draw[red] (1,0) node{\tiny$\bullet$};
\draw[red] (3,1) node{\tiny$\bullet$};
\draw[red] (3,2) node{\tiny$\bullet$};
\draw (1.5,-0.8) node{(a)};

\end{scope}

\begin{scope}[shift={(4.5,0)}]

\draw (1,1)--(2,1);

\draw[red] (2.2,2)--(3.2,2) node[midway, above] {$e_1$};
\draw[red] (2,1)--(3,1) node[midway, below] {$e_1$};

\draw (1,1)--(0.8,0);

\draw[green!70!black] (1.2,2)--(2.2,2) node[midway, above] {$e_2$};
\draw[green!70!black] (0.8,0)--(1.8,0) node[midway, below] {$e_2$};

\draw (2,1)--(2.2,2);

\draw[blue] (0,1)--(1,1) node[midway, above] {$e_3$};
\draw[blue] (-0.2,0)--(0.8,0) node[midway, below] {$e_3$};

\draw[brown] (-0.2,0)--(0,1) node[midway, left] {$e_4$};
\draw[brown] (1.8,0)--(2,1) node[midway, right] {$e_4$};

\draw[yellow!80!black] (1,1)--(1.2,2) node[midway, left] {$e_5$};
\draw[yellow!80!black] (3,1)--(3.2,2) node[midway, right] {$e_5$};

\draw (-0.2,0) node{\tiny$\bullet$};
\draw (0,1) node{\tiny$\bullet$};
\draw (2.2,2) node{\tiny$\bullet$};
\draw (2,1) node{\tiny$\bullet$};
\draw (1.8,0) node{\tiny$\bullet$};

\draw[red] (1.2,2) node{\tiny$\bullet$};
\draw[red] (1,1) node{\tiny$\bullet$};
\draw[red] (0.8,0) node{\tiny$\bullet$};
\draw[red] (3,1) node{\tiny$\bullet$};
\draw[red] (3.2,2) node{\tiny$\bullet$};

\draw (1.5,-0.8) node{(b)};

\end{scope}

\begin{scope}[shift={(9,0)}]

\draw (1, 0.866)--(2, 0.866);
\draw (0.5, 0)--(1, 0.866);
\draw (2, 0.866)--(2.5, 1.732);

\draw[dashed, gray, thick] (0, 0.866)--(0.5, 0);
\draw[dashed, gray, thick] (1, 0.866)--(1.5, 0);
\draw[dashed, gray, thick] (1.5, 1.732)--(2, 0.866);
\draw[dashed, gray, thick] (2.5, 1.732)--(3, 0.866);

\draw[red] (2.5, 1.732)--(3.5, 1.732) node[midway, above] {$e_1$};
\draw[red] (2, 0.866)--(3, 0.866) node[midway, below] {$e_1$};

\draw[green!70!black] (1.5, 1.732)--(2.5, 1.732) node[midway, above] {$e_2$};
\draw[green!70!black] (0.5, 0)--(1.5, 0) node[midway, below] {$e_2$};

\draw[blue] (0, 0.866)--(1, 0.866) node[midway, above] {$e_3$};
\draw[blue] (-0.5, 0)--(0.5, 0) node[midway, below] {$e_3$};

\draw[brown] (-0.5, 0)--(0, 0.866) node[midway, left] {$e_4$};
\draw[brown] (1.5, 0)--(2, 0.866) node[midway, right] {$e_4$};

\draw[yellow!80!black] (1, 0.866)--(1.5, 1.732) node[midway, left] {$e_5$};
\draw[yellow!80!black] (3, 0.866)--(3.5, 1.732) node[midway, right] {$e_5$};

\draw (-0.5, 0) node{\tiny$\bullet$};
\draw (0, 0.866) node{\tiny$\bullet$};
\draw (2.5, 1.732) node{\tiny$\bullet$};
\draw (2, 0.866) node{\tiny$\bullet$};
\draw (1.5, 0) node{\tiny$\bullet$};

\draw[red] (1.5, 1.732) node{\tiny$\bullet$};
\draw[red] (1, 0.866) node{\tiny$\bullet$};
\draw[red] (0.5, 0) node{\tiny$\bullet$};
\draw[red] (3, 0.866) node{\tiny$\bullet$};
\draw[red] (3.5, 1.732) node{\tiny$\bullet$};

\draw (1.5, -0.8) node{(c)};

\end{scope}

\end{tikzpicture}
\caption{(a) $X_{\mathrm{reg}}\in\mathcal{P}_2$, (b) $(X_{\mathrm{reg}})_\epsilon$, and (c) Maximal surface in $\mathcal{H}(1,1)$.}
\label{sq-tiled deform}
\end{figure}

This produces a translation surface $(X_{\mathrm{reg}})_\epsilon$ in the same stratum $\mathcal{H}(k_1,\dots,k_n)$ (For instance, see Figure~\ref{sq-tiled deform} (b)), and each systolic connection coming from a side of a rhombus still has length~$1$.  
Since $\mathrm{area}(S_i^\epsilon)=\cos\epsilon$ for $\epsilon\in\left[0,\frac{\pi}{6}\right]$, we have
\[
\mathrm{area}((X_{\mathrm{reg}})_{\epsilon})= r\cos\epsilon > r\cos\epsilon'= \mathrm{area}((X_{\mathrm{reg}})_{\epsilon'}), 
\]
for $\epsilon<\epsilon'\in[0,\frac{\pi}{6}]$.

After normalizing both surfaces to have area~$1$, denote the resulting
surfaces by $(\tilde{X}_{\mathrm{reg}})_\epsilon$ and
$(\tilde{X}_{\mathrm{reg}})_{\epsilon'}$. Their systoles are
respectively
\[
\frac{1}{\sqrt{\operatorname{area}((X_{\mathrm{reg}})_\epsilon)}}
\quad\text{and}\quad
\frac{1}{\sqrt{\operatorname{area}((X_{\mathrm{reg}})_{\epsilon'})}}.
\]
Since
$\operatorname{area}((X_{\mathrm{reg}})_\epsilon)
>
\operatorname{area}((X_{\mathrm{reg}})_{\epsilon'})$,
for $\epsilon<\epsilon'$, it follows that
\[
\operatorname{sys}\left((\tilde{X}_{\mathrm{reg}})_\epsilon\right)
<
\operatorname{sys}\left((\tilde{X}_{\mathrm{reg}})_{\epsilon'}\right).
\]
Hence, the deformation is systole-increasing. For simplicity, we
continue to denote the normalized surfaces by
$(X_{\mathrm{reg}})_\epsilon$ and $(X_{\mathrm{reg}})_{\epsilon'}$,
respectively.
Define
\[
\widetilde{\gamma}:[0,\tfrac{\pi}{6}] \longrightarrow \mathcal{H}(k_1,\dots,k_n),
\ \text{by}\quad
\widetilde{\gamma}(\epsilon)=(X_{\mathrm{reg}})_\epsilon.
\]
The construction of $(X_{\mathrm{reg}})_\epsilon$ varies continuously with $\epsilon$, so $\widetilde{\gamma}$ is continuous.  
Moreover, the above monotonicity shows that $\mathrm{sys}(\widetilde{\gamma}(\epsilon))$ is strictly increasing in~$\epsilon$.
At $\epsilon=\frac{\pi}{6}$, each rhombus has an acute angle of $\frac{\pi}{3}$ and therefore decomposes naturally into a pair of equilateral triangles.  
By the characterization of maximal surfaces stated in Theorem~\ref{maximal surfaces}, a translation surface tiled by equilateral triangles is maximal.  
Therefore $\widetilde{\gamma}(\frac{\pi}{6})$ is a maximal surface in $\mathcal{H}(k_1,\dots,k_n)$. For instance, applying this deformation to the initial surface in Figure~\ref{sq-tiled deform}(a) yields a maximal surface in $\mathcal{H}(1,1)$ illustrated in Figure~\ref{sq-tiled deform}(c). 

Now define $\phi:[0,1]\longrightarrow\left[0,\frac{\pi}{6}\right],
\ \text{by}\
\phi(t)=\frac{\pi}{6}t,$
and 
\[
\gamma:[0,1]\longrightarrow\mathcal{H}(k_1,\dots,k_n),
\ \text{by}\
\gamma(t) = (\widetilde{\gamma} \circ \phi)(t).
\]
Since both $\phi$ and $\widetilde{\gamma}$ are continuous, $\gamma$ is continuous. Moreover, $\phi$ is strictly increasing, so the monotonicity of the systole is preserved. Hence
\begin{enumerate}
    \item $\gamma(0)=X_{\mathrm{reg}}$ and $\gamma(1)$ is a maximal surface, and
    \item $\operatorname{sys}\left(\gamma\left(t'\right)\right)
    <
    \operatorname{sys}\left(\gamma\left(t\right)\right),
    \ \text{for all } 0\le t'<t\le1.$
\end{enumerate}

This proves the result for case $1$.

\noindent\textbf{Case 2:} $X_{\mathrm{reg}}\in\displaystyle\bigcup_{k\ge 3}\mathcal{P}_k$.

Before proceeding with this case, we establish the following proposition (Proposition~\ref{even-gon-deform}) and lemmas (Lemmas~\ref{even-gon-area} and~\ref{even-gon-monotonicity}), which will be used in the proof.

\begin{proposition}\label{even-gon-deform}
For every integer $k\ge 3$, there exists a continuous family
of equilateral $2k$-gons $\{P_t^k\}_{t\in[0,1]}$, such that:
\begin{enumerate}
\item Each $P_t^k$ has side length $1$.
    \item $P_0^k$ is a regular $2k$-gon.
    \item $P_1^k$ is the degenerate $2k$ gon given by the boundary
    of $k-1$ adjacent unit squares, with collinear vertices
    retained.
    \item For every $t\in[0,1]$, each pair of opposite sides
    of $P_t^k$ remains parallel.
\end{enumerate}
\end{proposition}

\begin{proof}
We explicitly construct the family of polygons by parameterizing the side directions with respect to $t \in [0,1]$. For each $t\in[0,1]$, set $x(t)=\frac{(1-t)\pi}{2k}.$ Now, for each $t\in[0,1]$, define the angles $\theta_i(t), i=1,2,\dots,k$ by
\begin{align*}
\theta_i(t) &= 2(i-1)x(t), \quad i=1,\ldots,k-1,\\
\theta_k(t) &= \frac{\pi}{2}+(k-2)x(t).
\end{align*}

We define $P_t^k$ to be the equilateral $2k$-gon of side length $1$ whose vertices, in cyclic order, are $p_i(t)$, for $i=0,1,\ldots,2k-1$, with coordinates given by
\begin{align*}
p_0(t) &= (0,0),\\
p_j(t) &=
\left(
\sum_{i=1}^{j}\cos\theta_i(t),\,
\sum_{i=1}^{j}\sin\theta_i(t)
\right),
\quad j=1,\ldots,k, \, \text{and}\\
p_{k+j}(t) &=
\left(
\sum_{i=j+1}^{k}\cos\theta_i(t),\,
\sum_{i=j+1}^{k}\sin\theta_i(t)
\right),
\quad j=1,\ldots,k-1.
\end{align*}

By construction, $\theta_i(t)$ is the direction angle of the side
$p_{i-1}(t)p_i(t)$ of $P_t^k$, measured counterclockwise from the
positive $x$-axis, for each $i=1,\ldots,k$.
See Figures~\ref{fig:octagon_deformation_k4}(B) and~\ref{fig:2k_gon_all_states}(B) for $k=4$ and $k=5$, respectively.

Let $\vec{v}_i(t)=\bigl(\cos\theta_i(t),\sin\theta_i(t)\bigr)$, $i=1,\ldots,k$. Thus, the successive side vectors of $P_t^k$ are $\vec{v}_1(t),\ldots,\vec{v}_k(t),
-\vec{v}_1(t),\ldots,-\vec{v}_k(t)$.

Hence every side has length $1$, and each pair of opposite
sides remains parallel. Since the vertices depend continuously
on $t$, the family $\{P_t^k\}_{t\in[0,1]}$ is continuous.

\begin{figure}[htbp]
    \centering
    \begin{subfigure}[b]{0.32\textwidth}
        \centering
        \resizebox{\linewidth}{!}{%
        \begin{tikzpicture}[
            thick,
            decoration={markings, mark=at position 0.55 with {\arrow{stealth}}}
        ]
            \def\thOne{0}
            \def\thTwo{45}
            \def\thThree{90}
            \def\thFour{135}

            \draw[-stealth, semithick, gray] (-1.5, 0) -- (4.0, 0) node[right, text=black] {$x$};
            \draw[-stealth, semithick, gray] (0, -0.5) -- (0, 3.5) node[above, text=black] {$y$};

            \coordinate (P0) at (0,0);
            \path (P0) ++(\thOne:1) coordinate (P1);
            \path (P1) ++(\thTwo:1) coordinate (P2);
            \path (P2) ++(\thThree:1) coordinate (P3);
            \path (P3) ++(\thFour:1) coordinate (P4);
            \path (P4) ++(180+\thOne:1) coordinate (P5);
            \path (P5) ++(180+\thTwo:1) coordinate (P6);
            \path (P6) ++(180+\thThree:1) coordinate (P7);

            \draw[postaction={decorate}] (P0) -- (P1);
            \draw[postaction={decorate}] (P1) -- (P2);
            \draw[postaction={decorate}] (P2) -- (P3);
            \draw[postaction={decorate}] (P3) -- (P4);
            \draw[postaction={decorate}] (P4) -- (P5);
            \draw[postaction={decorate}] (P5) -- (P6);
            \draw[postaction={decorate}] (P6) -- (P7);
            \draw[postaction={decorate}] (P7) -- (P0);

            \draw[dashed, thin, gray] (P1) -- ++(0.7,0);
            \draw[-stealth, thin] (P1) +(0.35,0) arc [start angle=0, end angle=\thTwo, radius=0.35];
            \node[font=\tiny, fill=white, inner sep=0.5pt] at ([shift={({\thTwo/2}:0.55)}]P1) {\(\frac{\pi}{4}\)};

            \draw[dashed, thin, gray] (P2) -- ++(0.7,0);
            \draw[-stealth, thin] (P2) +(0.35,0) arc [start angle=0, end angle=\thThree, radius=0.35];
            \node[font=\tiny, fill=white, inner sep=0.5pt] at ([shift={({\thThree/2}:0.55)}]P2) {\(\frac{\pi}{2}\)};

            \draw[dashed, thin, gray] (P3) -- ++(0.7,0);
            \draw[-stealth, thin] (P3) +(0.35,0) arc [start angle=0, end angle=\thFour, radius=0.35];
            \node[font=\tiny, fill=white, inner sep=0.5pt] at ([shift={({\thFour/2}:0.55)}]P3) {\(\frac{3\pi}{4}\)};

            \fill (P0) circle (1.5pt) node[below left, font=\tiny] {\(p_0(0)\)};
            \fill (P1) circle (1.5pt) node[below right, font=\tiny] {\(p_1(0)\)};
            \fill (P2) circle (1.5pt) node[below right=1pt, xshift=2pt, font=\tiny] {\(p_2(0)\)}; 
            \fill (P3) circle (1.5pt) node[above right=1pt, xshift=5pt, font=\tiny] {\(p_3(0)\)}; 
            \fill (P4) circle (1.5pt) node[above, font=\tiny] {\(p_4(0)\)};
            \fill (P5) circle (1.5pt) node[above left, font=\tiny] {\(p_5(0)\)};
            \fill (P6) circle (1.5pt) node[left, font=\tiny] {\(p_6(0)\)};
            \fill (P7) circle (1.5pt) node[below left, font=\tiny] {\(p_7(0)\)};
        \end{tikzpicture}}
        \caption{$P_t^4$, $t=0.$}
    \end{subfigure}
    \hfill
    \begin{subfigure}[b]{0.32\textwidth}
        \centering
        \resizebox{\linewidth}{!}{%
        \begin{tikzpicture}[
            thick,
            decoration={markings, mark=at position 0.55 with {\arrow{stealth}}}
        ]
            \def\thOne{0}
            \def\thTwo{27}
            \def\thThree{54}
            \def\thFour{117}

            \draw[-stealth, semithick, gray] (-1.5, 0) -- (4.0, 0) node[right, text=black] {$x$};
            \draw[-stealth, semithick, gray] (0, -0.5) -- (0, 3.5) node[above, text=black] {$y$};

            \coordinate (P0) at (0,0);
            \path (P0) ++(\thOne:1) coordinate (P1);
            \path (P1) ++(\thTwo:1) coordinate (P2);
            \path (P2) ++(\thThree:1) coordinate (P3);
            \path (P3) ++(\thFour:1) coordinate (P4);
            \path (P4) ++(180+\thOne:1) coordinate (P5);
            \path (P5) ++(180+\thTwo:1) coordinate (P6);
            \path (P6) ++(180+\thThree:1) coordinate (P7);

            \draw[postaction={decorate}] (P0) -- (P1);
            \draw[postaction={decorate}] (P1) -- (P2);
            \draw[postaction={decorate}] (P2) -- (P3);
            \draw[postaction={decorate}] (P3) -- (P4);
            \draw[postaction={decorate}] (P4) -- (P5);
            \draw[postaction={decorate}] (P5) -- (P6);
            \draw[postaction={decorate}] (P6) -- (P7);
            \draw[postaction={decorate}] (P7) -- (P0);

            \draw[dashed, thin, gray] (P1) -- ++(0.7,0);
            \draw[-stealth, thin] (P1) +(0.4,0) arc [start angle=0, end angle=\thTwo, radius=0.4];
            \node[font=\tiny, fill=white, inner sep=0.5pt] at ([shift={({\thTwo/2}:0.65)}]P1) {\(\theta_2(t)\)};

            \draw[dashed, thin, gray] (P2) -- ++(0.7,0);
            \draw[-stealth, thin] (P2) +(0.4,0) arc [start angle=0, end angle=\thThree, radius=0.4];
            \node[font=\tiny, fill=white, inner sep=0.5pt, xshift=3pt] at ([shift={({\thThree/2}:0.65)}]P2) {\(\theta_3(t)\)};

            \draw[dashed, thin, gray] (P3) -- ++(0.7,0);
            \draw[-stealth, thin] (P3) +(0.4,0) arc [start angle=0, end angle=\thFour, radius=0.4];
            \node[font=\tiny, fill=white, inner sep=0.5pt] at ([shift={({\thFour/2-10}:0.65)}]P3) {\(\theta_4(t)\)};

            \fill (P0) circle (1.5pt) node[below left, font=\tiny] {\(p_0(t)\)};
            \fill (P1) circle (1.5pt) node[below right, font=\tiny] {\(p_1(t)\)};
            \fill (P2) circle (1.5pt) node[below right=2pt, xshift=2pt, font=\tiny] {\(p_2(t)\)}; 
            \fill (P3) circle (1.5pt) node[below right, font=\tiny] {\(p_3(t)\)}; 
            \fill (P4) circle (1.5pt) node[above, font=\tiny] {\(p_4(t)\)};
            \fill (P5) circle (1.5pt) node[above left, font=\tiny] {\(p_5(t)\)};
            \fill (P6) circle (1.5pt) node[left, font=\tiny] {\(p_6(t)\)};
            \fill (P7) circle (1.5pt) node[below left, font=\tiny] {\(p_7(t)\)};
        \end{tikzpicture}}
        \caption{$P_t^4$, $0<t<1$.}
    \end{subfigure}
    \hfill
    \begin{subfigure}[b]{0.32\textwidth}
        \centering
        \resizebox{\linewidth}{!}{%
        \begin{tikzpicture}[
            thick,
            decoration={markings, mark=at position 0.55 with {\arrow{stealth}}}
        ]
            \def\thOne{0}
            \def\thTwo{0}
            \def\thThree{0}
            \def\thFour{90}

            \draw[-stealth, semithick, gray] (-1.5, 0) -- (4.0, 0) node[right, text=black] {$x$};
            \draw[-stealth, semithick, gray] (0, -0.5) -- (0, 3.5) node[above, text=black] {$y$};

            \coordinate (P0) at (0,0);
            \path (P0) ++(\thOne:1) coordinate (P1);
            \path (P1) ++(\thTwo:1) coordinate (P2);
            \path (P2) ++(\thThree:1) coordinate (P3);
            \path (P3) ++(\thFour:1) coordinate (P4);
            \path (P4) ++(180+\thOne:1) coordinate (P5);
            \path (P5) ++(180+\thTwo:1) coordinate (P6);
            \path (P6) ++(180+\thThree:1) coordinate (P7);

            \draw[postaction={decorate}] (P0) -- (P1);
            \draw[postaction={decorate}] (P1) -- (P2);
            \draw[postaction={decorate}] (P2) -- (P3);
            \draw[postaction={decorate}] (P3) -- (P4);
            \draw[postaction={decorate}] (P4) -- (P5);
            \draw[postaction={decorate}] (P5) -- (P6);
            \draw[postaction={decorate}] (P6) -- (P7);
            \draw[postaction={decorate}] (P7) -- (P0);

            \draw[densely dotted, thick, darkgray] (P1) -- (P6);
            \draw[densely dotted, thick, darkgray] (P2) -- (P5);

            \draw[dashed, thin, gray] (P3) -- ++(0.7,0);
            \draw[-stealth, thin] (P3) +(0.35,0) arc [start angle=0, end angle=\thFour, radius=0.35];
            \node[font=\tiny, fill=white, inner sep=0.5pt] at ([shift={(45:0.55)}]P3) {\(\frac{\pi}{2}\)};

            \fill (P0) circle (1.5pt) node[below=2pt, font=\tiny] {\(p_0(1)\)};
            \fill (P1) circle (1.5pt) node[below=2pt, font=\tiny] {\(p_1(1)\)};
            \fill (P2) circle (1.5pt) node[below=2pt, font=\tiny] {\(p_2(1)\)};
            \fill (P3) circle (1.5pt) node[below right, font=\tiny] {\(p_3(1)\)};
            \fill (P4) circle (1.5pt) node[above right, xshift=8pt, font=\tiny] {\(p_4(1)\)};
            \fill (P5) circle (1.5pt) node[above=2pt, font=\tiny] {\(p_5(1)\)};
            \fill (P6) circle (1.5pt) node[above=2pt, font=\tiny] {\(p_6(1)\)};
            \fill (P7) circle (1.5pt) node[above left, font=\tiny] {\(p_7(1)\)};
        \end{tikzpicture}}
        \caption{$P_t^4$, \(t=1\).}
    \end{subfigure}
    
    \caption{Equilateral octagon $P_t^4$ at (A) \(t=0\), (B) $0<t<1$, and (C) \(t=1\).}
    \label{fig:octagon_deformation_k4}
\end{figure}

\begin{figure}[htbp]
    \centering
    \begin{subfigure}[b]{0.32\textwidth}
        \centering
        \resizebox{\linewidth}{!}{%
        \begin{tikzpicture}[
            thick,
            decoration={markings, mark=at position 0.55 with {\arrow{stealth}}}
        ]
            \def\thOne{0}
            \def\thTwo{36}
            \def\thThree{72}
            \def\thFour{108}
            \def\thFive{144}

            \draw[-stealth, semithick, gray] (-1.5, 0) -- (4.8, 0) node[right, text=black] {$x$};
            \draw[-stealth, semithick, gray] (0, -0.5) -- (0, 3.8) node[above, text=black] {$y$};

            \coordinate (P0) at (0,0);
            \path (P0) ++(\thOne:1) coordinate (P1);
            \path (P1) ++(\thTwo:1) coordinate (P2);
            \path (P2) ++(\thThree:1) coordinate (P3);
            \path (P3) ++(\thFour:1) coordinate (P4);
            \path (P4) ++(\thFive:1) coordinate (P5);
            \path (P5) ++(180+\thOne:1) coordinate (P6);
            \path (P6) ++(180+\thTwo:1) coordinate (P7);
            \path (P7) ++(180+\thThree:1) coordinate (P8);
            \path (P8) ++(180+\thFour:1) coordinate (P9);

            \draw[postaction={decorate}] (P0) -- (P1);
            \draw[postaction={decorate}] (P1) -- (P2);
            \draw[postaction={decorate}] (P2) -- (P3);
            \draw[postaction={decorate}] (P3) -- (P4);
            \draw[postaction={decorate}] (P4) -- (P5);
            \draw[postaction={decorate}] (P5) -- (P6);
            \draw[postaction={decorate}] (P6) -- (P7);
            \draw[postaction={decorate}] (P7) -- (P8);
            \draw[postaction={decorate}] (P8) -- (P9);
            \draw[postaction={decorate}] (P9) -- (P0);

            \draw[dashed, thin, gray] (P1) -- ++(0.7,0);
            \draw[-stealth, thin] (P1) +(0.35,0) arc [start angle=0, end angle=\thTwo, radius=0.35];
            \node[font=\tiny, fill=white, inner sep=0.5pt] at ([shift={({\thTwo/2}:0.55)}]P1) {\(\frac{\pi}{5}\)};

            \draw[dashed, thin, gray] (P2) -- ++(0.7,0);
            \draw[-stealth, thin] (P2) +(0.35,0) arc [start angle=0, end angle=\thThree, radius=0.35];
            \node[font=\tiny, fill=white, inner sep=0.5pt] at ([shift={({\thThree/2}:0.55)}]P2) {\(\frac{2\pi}{5}\)};

            \draw[dashed, thin, gray] (P4) -- ++(0.7,0);
            \draw[-stealth, thin] (P4) +(0.35,0) arc [start angle=0, end angle=\thFive, radius=0.35];
            \node[font=\tiny, fill=white, inner sep=0.5pt] at ([shift={({\thFive/2}:0.55)}]P4) {\(\frac{4\pi}{5}\)};

            \fill (P0) circle (1.5pt) node[below left, font=\tiny] {\(p_0(0)\)};
            \fill (P1) circle (1.5pt) node[below, font=\tiny] {\(p_1(0)\)};
            \fill (P2) circle (1.5pt) node[below right, font=\tiny] {\(p_2(0)\)};
            \fill (P3) circle (1.5pt) node[right, font=\tiny] {\(p_3(0)\)};
            \fill (P4) circle (1.5pt) node[below right=1pt, font=\tiny] {\(p_4(0)\)};
            \fill (P5) circle (1.5pt) node[above, font=\tiny] {\(p_5(0)\)};
            \fill (P6) circle (1.5pt) node[above left, font=\tiny] {\(p_6(0)\)};
            \fill (P7) circle (1.5pt) node[left, font=\tiny] {\(p_7(0)\)};
            \fill (P8) circle (1.5pt) node[left, font=\tiny] {\(p_8(0)\)};
            \fill (P9) circle (1.5pt) node[below left, font=\tiny] {\(p_9(0)\)};
        \end{tikzpicture}}
        \caption{$P_t^5$, $t=0.$}
    \end{subfigure}
    \hfill
    \begin{subfigure}[b]{0.32\textwidth}
        \centering
        \resizebox{\linewidth}{!}{%
        \begin{tikzpicture}[
            thick,
            decoration={markings, mark=at position 0.55 with {\arrow{stealth}}}
        ]
            \def\thOne{0}
            \def\thTwo{21.6}
            \def\thThree{43.2}
            \def\thFour{64.8}
            \def\thFive{122.4}

            \draw[-stealth, semithick, gray] (-1.5, 0) -- (4.8, 0) node[right, text=black] {$x$};
            \draw[-stealth, semithick, gray] (0, -0.5) -- (0, 3.8) node[above, text=black] {$y$};

            \coordinate (P0) at (0,0);
            \path (P0) ++(\thOne:1) coordinate (P1);
            \path (P1) ++(\thTwo:1) coordinate (P2);
            \path (P2) ++(\thThree:1) coordinate (P3);
            \path (P3) ++(\thFour:1) coordinate (P4);
            \path (P4) ++(\thFive:1) coordinate (P5);
            \path (P5) ++(180+\thOne:1) coordinate (P6);
            \path (P6) ++(180+\thTwo:1) coordinate (P7);
            \path (P7) ++(180+\thThree:1) coordinate (P8);
            \path (P8) ++(180+\thFour:1) coordinate (P9);

            \draw[postaction={decorate}] (P0) -- (P1);
            \draw[postaction={decorate}] (P1) -- (P2);
            \draw[postaction={decorate}] (P2) -- (P3);
            \draw[postaction={decorate}] (P3) -- (P4);
            \draw[postaction={decorate}] (P4) -- (P5);
            \draw[postaction={decorate}] (P5) -- (P6);
            \draw[postaction={decorate}] (P6) -- (P7);
            \draw[postaction={decorate}] (P7) -- (P8);
            \draw[postaction={decorate}] (P8) -- (P9);
            \draw[postaction={decorate}] (P9) -- (P0);

            \draw[dashed, thin, gray] (P1) -- ++(0.7,0);
            \draw[-stealth, thin] (P1) +(0.4,0) arc [start angle=0, end angle=\thTwo, radius=0.4];
            \node[font=\tiny, fill=white, inner sep=0.5pt] at ([shift={({\thTwo/2}:0.65)}]P1) {\(\theta_2(t)\)};

            \draw[dashed, thin, gray] (P2) -- ++(0.7,0);
            \draw[-stealth, thin] (P2) +(0.4,0) arc [start angle=0, end angle=\thThree, radius=0.4];
            \node[font=\tiny, fill=white, inner sep=0.5pt] at ([shift={({\thThree/2}:0.65)}]P2) {\(\theta_3(t)\)};

            \draw[dashed, thin, gray] (P4) -- ++(0.7,0);
            \draw[-stealth, thin] (P4) +(0.4,0) arc [start angle=0, end angle=\thFive, radius=0.4];
            \node[font=\tiny, fill=white, inner sep=0.5pt] at ([shift={({\thFive/2-5}:0.65)}]P4) {\(\theta_5(t)\)};

            \fill (P0) circle (1.5pt) node[below left, font=\tiny] {\(p_0(t)\)};
            \fill (P1) circle (1.5pt) node[below, font=\tiny] {\(p_1(t)\)};
            \fill (P2) circle (1.5pt) node[below right, font=\tiny] {\(p_2(t)\)};
            \fill (P3) circle (1.5pt) node[right, font=\tiny] {\(p_3(t)\)};
            \fill (P4) circle (1.5pt) node[below right=1pt, font=\tiny] {\(p_4(t)\)};
            \fill (P5) circle (1.5pt) node[above, font=\tiny] {\(p_5(t)\)};
            \fill (P6) circle (1.5pt) node[above left, font=\tiny] {\(p_6(t)\)};
            \fill (P7) circle (1.5pt) node[left, font=\tiny] {\(p_7(t)\)};
            \fill (P8) circle (1.5pt) node[left, font=\tiny] {\(p_8(t)\)};
            \fill (P9) circle (1.5pt) node[below left, font=\tiny] {\(p_9(t)\)};
        \end{tikzpicture}}
        \caption{$P_t^5$, $0<t<1$.}
    \end{subfigure}
    \hfill
    \begin{subfigure}[b]{0.32\textwidth}
        \centering
        \resizebox{\linewidth}{!}{%
        \begin{tikzpicture}[
            thick,
            decoration={markings, mark=at position 0.55 with {\arrow{stealth}}}
        ]
            \def\thOne{0}
            \def\thTwo{0}
            \def\thThree{0}
            \def\thFour{0}
            \def\thFive{90}

            \draw[-stealth, semithick, gray] (-1.5, 0) -- (4.8, 0) node[right, text=black] {$x$};
            \draw[-stealth, semithick, gray] (0, -0.5) -- (0, 3.8) node[above, text=black] {$y$};

            \coordinate (P0) at (0,0);
            \path (P0) ++(\thOne:1) coordinate (P1);
            \path (P1) ++(\thTwo:1) coordinate (P2);
            \path (P2) ++(\thThree:1) coordinate (P3);
            \path (P3) ++(\thFour:1) coordinate (P4);
            \path (P4) ++(\thFive:1) coordinate (P5);
            \path (P5) ++(180+\thOne:1) coordinate (P6);
            \path (P6) ++(180+\thTwo:1) coordinate (P7);
            \path (P7) ++(180+\thThree:1) coordinate (P8);
            \path (P8) ++(180+\thFour:1) coordinate (P9);

            \draw[postaction={decorate}] (P0) -- (P1);
            \draw[postaction={decorate}] (P1) -- (P2);
            \draw[postaction={decorate}] (P2) -- (P3);
            \draw[postaction={decorate}] (P3) -- (P4);
            \draw[postaction={decorate}] (P4) -- (P5);
            \draw[postaction={decorate}] (P5) -- (P6);
            \draw[postaction={decorate}] (P6) -- (P7);
            \draw[postaction={decorate}] (P7) -- (P8);
            \draw[postaction={decorate}] (P8) -- (P9);
            \draw[postaction={decorate}] (P9) -- (P0);

            \draw[densely dotted, thick, darkgray] (P1) -- (P8);
            \draw[densely dotted, thick, darkgray] (P2) -- (P7);
            \draw[densely dotted, thick, darkgray] (P3) -- (P6);

            \draw[dashed, thin, gray] (P4) -- ++(0.7,0);
            \draw[-stealth, thin] (P4) +(0.35,0) arc [start angle=0, end angle=\thFive, radius=0.35];
            \node[font=\tiny, fill=white, inner sep=0.5pt] at ([shift={(45:0.55)}]P4) {\(\frac{\pi}{2}\)};

            \fill (P0) circle (1.5pt) node[below=2pt, font=\tiny] {\(p_0(1)\)};
            \fill (P1) circle (1.5pt) node[below=2pt, font=\tiny] {\(p_1(1)\)};
            \fill (P2) circle (1.5pt) node[below=2pt, font=\tiny] {\(p_2(1)\)};
            \fill (P3) circle (1.5pt) node[below=2pt, font=\tiny] {\(p_3(1)\)};
            \fill (P4) circle (1.5pt) node[below right, font=\tiny] {\(p_4(1)\)};
            \fill (P5) circle (1.5pt) node[above right, font=\tiny] {\(p_5(1)\)};
            \fill (P6) circle (1.5pt) node[above=2pt, font=\tiny] {\(p_6(1)\)};
            \fill (P7) circle (1.5pt) node[above=2pt, font=\tiny] {\(p_7(1)\)};
            \fill (P8) circle (1.5pt) node[above=2pt, font=\tiny] {\(p_8(1)\)};
            \fill (P9) circle (1.5pt) node[above left, font=\tiny] {\(p_9(1)\)};
        \end{tikzpicture}}
        \caption{$P_t^5$, \(t=1\).}
    \end{subfigure}
    
    \caption{Equilateral decagon $P_t^5$ at (A) \(t=0\), (B) $0<t<1$, and (C) \(t=1\).}
    \label{fig:2k_gon_all_states}
\end{figure}

At $t=0$, we have
\[
\theta_i(0)=\frac{(i-1)\pi}{k},
\qquad i=1,\ldots,k.
\]
Thus the successive side directions differ by $\pi/k$,
so $P_0^k$ is a regular $2k$-gon. See Figures \ref{fig:octagon_deformation_k4}(A) and \ref{fig:2k_gon_all_states}(A) for $k=4$ and $5$, respectively.

At $t=1$, we have
\[
\theta_i(1)=0,\quad i=1,\ldots,k-1,
\quad
\theta_k(1)=\frac{\pi}{2}.
\]
So, the coordinates of $P_1^k$ are 
\begin{align*}
p_0(1) &= (0,0),\\
p_j(1) &= (j,0), \quad j=1,\ldots,k-1,\\
p_k(1) &= (k-1,1),\\
p_{k+j}(1) &= (k-1-j,1), \quad j=1,\ldots,k-1.
\end{align*}
Consequently, $P_1^k$ is the boundary of a
$(k-1)\times 1$ rectangle, with collinear vertices retained.
This is precisely the boundary of $k-1$ adjacent unit squares. See Figures \ref{fig:octagon_deformation_k4}(C) and \ref{fig:2k_gon_all_states}(C) for $k=4$ and $5$, respectively.

This completes the proof of the proposition.

\end{proof}

\begin{lemma}\label{even-gon-area}
Let $P_t^k$, $0\le t\le 1$, $k\ge 3,$ be the family of equilateral
$2k$-gons defined in Proposition~\ref{even-gon-deform}.
Then
\[
\operatorname{area}(P_t^k)=A_k(x(t)),\ x(t)=\frac{(1-t)\pi}{2k}.
\]
where
\[
A_k(x)=\sum_{1\le i<j\le k}
\sin\bigl(\theta_j-\theta_i\bigr),
\]
with $\theta_i=2(i-1)x,\quad i=1,\ldots,k-1,
\quad
\theta_k=\frac{\pi}{2}+(k-2)x.$
Equivalently,
\[
A_k(x)=\sum_{d=1}^{k-2}(k-1-d)\sin(2dx)+\sum_{j=0}^{k-2}\cos\bigl((k-2-2j)x\bigr).
\]
In particular,
\[
A_k(0)=k-1,\qquad
A_k\left(\frac{\pi}{2k}\right)
=\frac{k}{2}\cot\left(\frac{\pi}{2k}\right).
\]
\end{lemma}

\begin{proof}
Recall that a polygon is centrally symmetric if its boundary consists of pairs of opposite, parallel edges of equal length.
Since $P_t^k$ is centrally symmetric, its area is the sum of
the areas of the parallelograms generated by the pairs of side vectors
$\vec{v_i}(t), \vec{v_j}(t)$, where $1\le i<j\le k$ (applying Lemma $2$ in \cite{MR738519} for $d=2$). For instance, see Figure \ref{fig:octagon-area-decomposition-combined}. Therefore,
\[
\begin{aligned}
\operatorname{area}(P_t^k)
&=\sum_{1\le i<j\le k}
\left|\det(v_i(t),v_j(t))\right|\\
&=\sum_{1\le i<j\le k}
\sin\bigl(\theta_j(t)-\theta_i(t)\bigr).
\end{aligned}
\]

\begin{figure}[htpb]
    \centering
    \begin{subfigure}[b]{0.48\textwidth}
        \centering
        \begin{tikzpicture}[scale=1.6, line join=round]
            \pgfmathsetmacro{\s}{sqrt(2)/2}

            \tikzset{
                square/.style={fill=blue!15, draw=black, thick},
                rhombusV/.style={fill=green!15, draw=black, thick}, 
                rhombusH/.style={fill=red!15, draw=black, thick},    
                midarrow/.style={postaction={decorate, decoration={markings, mark=at position 0.55 with {\arrow{stealth}}}}}
            }

            \filldraw[square] (0,0) -- (1,0) -- (1,1) -- (0,1) -- cycle;
            \node at (0.5, 0.5) {$P_{\vec{v}_1\vec{v}_3}$};

            \filldraw[rhombusV] (1,0) -- (1+\s, \s) -- (1+\s, 1+\s) -- (1,1) -- cycle;
            \node at (1+\s/2, 0.5+\s/2) {$P_{\vec{v}_2\vec{v}_3}$};

            \filldraw[rhombusV] (0,0) -- (0,1) -- (-\s, 1+\s) -- (-\s, \s) -- cycle;
            \node at (-\s/2, 0.5+\s/2) {$P_{\vec{v}_3\vec{v}_4}$};

            \filldraw[rhombusH] (0,1) -- (1,1) -- (1-\s, 1+\s) -- (-\s, 1+\s) -- cycle;
            \node at (0.5-\s/2, 1+\s/2) {$P_{\vec{v}_1\vec{v}_4}$};

            \filldraw[square] (1,1) -- (1+\s, 1+\s) -- (1, 1+2*\s) -- (1-\s, 1+\s) -- cycle;
            \node at (1, 1+\s) {$P_{\vec{v}_2\vec{v}_4}$};

            \filldraw[rhombusH] (-\s, 1+\s) -- (1-\s, 1+\s) -- (1, 1+2*\s) -- (0, 1+2*\s) -- cycle;
            \node at (0.5-\s/2, 1+1.5*\s) {$P_{\vec{v}_1\vec{v}_2}$};

            \draw[midarrow, very thick] (0,0) -- (1,0) node[midway, below=3pt] {$\vec{v}_1$};
            \draw[midarrow, very thick] (1,0) -- (1+\s, \s) node[midway, below right=2pt] {$\vec{v}_2$};
            \draw[midarrow, very thick] (1+\s, \s) -- (1+\s, 1+\s) node[midway, right=3pt] {$\vec{v}_3$};
            \draw[midarrow, very thick] (1+\s, 1+\s) -- (1, 1+2*\s) node[midway, above right=2pt] {$\vec{v}_4$};

            \draw[midarrow, very thick] (1, 1+2*\s) -- (0, 1+2*\s) node[midway, above=3pt] {$-\vec{v}_1$};
            \draw[midarrow, very thick] (0, 1+2*\s) -- (-\s, 1+\s) node[midway, above left=2pt] {$-\vec{v}_2$};
            \draw[midarrow, very thick] (-\s, 1+\s) -- (-\s, \s) node[midway, left=3pt] {$-\vec{v}_3$};
            \draw[midarrow, very thick] (-\s, \s) -- (0,0) node[midway, below left=2pt] {$-\vec{v}_4$};

        \end{tikzpicture}
        \caption{$t = 0$ (Regular octagon)}
        \label{fig:octagon-decomp-t0}
    \end{subfigure}
    \hfill
    \begin{subfigure}[b]{0.48\textwidth}
        \centering
        \begin{tikzpicture}[scale=1.6, line join=round]
            \def\t{0.4}
            
            \pgfmathsetmacro{\x}{(1-\t)*180/8} 
            \pgfmathsetmacro{\thOne}{0}
            \pgfmathsetmacro{\thTwo}{2*\x}
            \pgfmathsetmacro{\thThree}{4*\x}
            \pgfmathsetmacro{\thFour}{90+2*\x}

            \coordinate (V1) at (\thOne:1);
            \coordinate (V2) at (\thTwo:1);
            \coordinate (V3) at (\thThree:1);
            \coordinate (V4) at (\thFour:1);

            \coordinate (O) at (0,0);
            \coordinate (V12) at ($(V1)+(V2)$);
            \coordinate (V13) at ($(V1)+(V3)$);
            \coordinate (V34) at ($(V3)+(V4)$);
            \coordinate (V123) at ($(V12)+(V3)$);
            \coordinate (V134) at ($(V13)+(V4)$);
            \coordinate (V1234) at ($(V123)+(V4)$);
            \coordinate (V234) at ($(V2)+(V3)+(V4)$);

            \tikzset{
                para13_24/.style={fill=blue!15, draw=black, thick}, 
                para23_34/.style={fill=green!15, draw=black, thick}, 
                para12_14/.style={fill=red!15, draw=black, thick},   
                midarrow/.style={postaction={decorate, decoration={markings, mark=at position 0.55 with {\arrow{stealth}}}}}
            }

            \filldraw[para13_24] (O) -- (V1) -- (V13) -- (V3) -- cycle;
            \node at ($(O)!0.5!(V13)$) {$P_{\vec{v}_1\vec{v}_3}$};

            \filldraw[para23_34] (V1) -- (V12) -- (V123) -- (V13) -- cycle;
            \node at ($(V1)!0.5!(V123)$) {\tiny$P_{\vec{v}_2\vec{v}_3}$};

            \filldraw[para23_34] (O) -- (V3) -- (V34) -- (V4) -- cycle;
            \node at ($(O)!0.5!(V34)$) {$P_{\vec{v}_3\vec{v}_4}$};

            \filldraw[para12_14] (V3) -- (V13) -- (V134) -- (V34) -- cycle;
            \node at ($(V3)!0.5!(V134)$) {$P_{\vec{v}_1\vec{v}_4}$};

            \filldraw[para13_24] (V13) -- (V123) -- (V1234) -- (V134) -- cycle;
            \node at ($(V13)!0.5!(V1234)$) {$P_{\vec{v}_2\vec{v}_4}$};

            \filldraw[para12_14] (V34) -- (V134) -- (V1234) -- (V234) -- cycle;
            \node at ($(V34)!0.5!(V1234)$) {$P_{\vec{v}_1\vec{v}_2}$};

            \draw[midarrow, very thick] (O) -- (V1) node[midway, below=3pt] {$\vec{v}_1$};
            \draw[midarrow, very thick] (V1) -- (V12) node[midway, below right=2pt] {$\vec{v}_2$};
            \draw[midarrow, very thick] (V12) -- (V123) node[midway, right=3pt] {$\vec{v}_3$};
            \draw[midarrow, very thick] (V123) -- (V1234) node[midway, above right=2pt] {$\vec{v}_4$};

            \draw[midarrow, very thick] (V1234) -- (V234) node[midway, above=3pt] {$-\vec{v}_1$};
            \draw[midarrow, very thick] (V234) -- (V34) node[midway, above left=2pt] {$-\vec{v}_2$};
            \draw[midarrow, very thick] (V34) -- (V4) node[midway, left=3pt] {$-\vec{v}_3$};
            \draw[midarrow, very thick] (V4) -- (O) node[midway, below left=2pt] {$-\vec{v}_4$};

        \end{tikzpicture}
        \caption{$t = 0.4$ (Intermediate state)}
        \label{fig:octagon-decomp-t04}
    \end{subfigure}
    
    \caption{Decomposition of the equilateral octagon $P_t^4$ into parallelograms $P_{\vec{v}_i\vec{v}_j}$ generated by its boundary edge vectors, shown at $t=0$ and $t=0.4$.}
    \label{fig:octagon-area-decomposition-combined}
\end{figure}

Here, the determinants are nonnegative because
$0\le\theta_i(t)\le\theta_j(t)<\pi$ for $i<j$.

For $1\le i<j\le k-1$, we have
\[
\theta_j-\theta_i=2(j-i)x.
\]

For the pairs with $1\le i<j\le k-1$, we group the terms according
to the difference $d=j-i$. For each fixed
$d\in\{1,\ldots,k-2\}$, there are exactly $k-1-d$ such pairs,
and each contributes $\sin(2dx)$. Hence, these pairs contribute
\[
\sum_{d=1}^{k-2}(k-1-d)\sin(2dx).
\]
The remaining terms correspond to pairs involving $\theta_k$.
Re-indexing these terms by $j=0,\ldots,k-2$, we obtain
\[
\sum_{j=0}^{k-2}\cos\bigl((k-2-2j)x\bigr).
\]
Combining the two contributions, we arrive at
\[
A_k(x)=\sum_{d=1}^{k-2}(k-1-d)\sin(2dx)
+\sum_{j=0}^{k-2}\cos\bigl((k-2-2j)x\bigr).
\]

The areas at the two endpoints are readily determined from the fact that \(P_1^k\) is a \((k-1)\times 1\) rectangle, while \(P_0^k\) is a regular \(2k\)-gon of unit side length.
\end{proof}

\begin{lemma}\label{even-gon-monotonicity}
Let $P_t^k$, $0\le t\le 1$, $k\ge 3,$ be the family of equilateral
$2k$-gons defined in Proposition~\ref{even-gon-deform}.
Then
\[
\operatorname{area}(P_{t_1}^k)
>
\operatorname{area}(P_{t_2}^k)
\quad
\text{for every }t_1<t_2\in[0,1].
\]
\end{lemma}

\begin{proof}
By Lemma~\ref{even-gon-area}, the area of $P_t^k$ is
$A_k(x(t))$, where $x(t)=\frac{(1-t)\pi}{2k}.$
Since $x(t)$ is strictly decreasing on $[0,1]$, it suffices
to show that $A_k(x)$ is strictly increasing on
$\left[0,\frac{\pi}{2k}\right]$.

Differentiating twice, we obtain
\[
A_k''(x)=-4\sum_{d=1}^{k-2}d^2(k-1-d)\sin(2dx)-\sum_{j=0}^{k-2}(k-2-2j)^2\cos\bigl((k-2-2j)x\bigr).
\]
For $0\le x\le\frac{\pi}{2k}$, we have
\[
0\le 2dx\le\frac{(k-2)\pi}{k}<\pi
\]
and
\[
\left|(k-2-2j)x\right|
\le\frac{(k-2)\pi}{2k}<\frac{\pi}{2}.
\]
Therefore, all the sine terms are nonnegative and all the
cosine terms are positive. Since at least one coefficient
in the second sum is nonzero, we conclude that
\[
A_k''(x)<0.
\]
It remains to show that
\[
A_k'\left(\frac{\pi}{2k}\right)=0.
\]
Put $n=k-1$. Using the standard trigonometric sum identities
\[
\sum_{j=1}^{n-1}(n-j)\sin(2jx)
=\frac{n\sin(2x)-\sin(2nx)}{4\sin^2x}
\]
and
\[
\sum_{j=0}^{n-1}\cos((n-1-2j)x)
=\frac{\sin(nx)}{\sin x},
\]
we obtain, for $x>0$,
\[
\begin{aligned}
A_k(x)
&=\frac{n\sin(2x)-\sin(2nx)}{4\sin^2x}
+\frac{\sin(nx)}{\sin x}\\
&=\frac{N(x)}{2\sin^2x},
\end{aligned}
\]
where
\[
N(x)=n\sin x\cos x
+\sin(nx)\bigl(2\sin x-\cos(nx)\bigr).
\]
Set $x_*=\frac{\pi}{2k}$. Since
$nx_*=\frac{\pi}{2}-x_*$, direct substitution gives
\[
N(x_*)=k\sin x_*\cos x_*,
\qquad
N'(x_*)=2k\cos^2x_*.
\]
Consequently,
\[
\begin{aligned}
A_k'(x_*)
&=\frac{N'(x_*)}{2\sin^2x_*}
-\frac{N(x_*)\cos x_*}{\sin^3x_*}\\
&=\frac{k\cos^2x_*}{\sin^2x_*}
-\frac{k\cos^2x_*}{\sin^2x_*}\\
&=0.
\end{aligned}
\]
Since $A_k''(x)<0$ throughout
$\left[0,\frac{\pi}{2k}\right]$, $A_k'(x)$
is strictly decreasing there. Hence
\[
A_k'(x)> A_k'\left(\frac{\pi}{2k}\right)=0
\quad
\text{for }\ 0\le x<\frac{\pi}{2k}.
\]
It follows that $A_k(x)$ is strictly increasing on
$\left[0,\frac{\pi}{2k}\right]$.

Finally, if $t_1<t_2$, then $x(t_1)>x(t_2)$, and therefore
\[
\operatorname{area}(P_{t_1}^k)
=A_k(x(t_1))
>A_k(x(t_2))
=\operatorname{area}(P_{t_2}^k).
\]
This proves the lemma.
\end{proof}

We are now ready to prove Theorem~\ref{main result 1} for Case $2$.

\textit{Proof for Case 2.}
Suppose $X_{\mathrm{reg}}$ is obtained by gluing regular
$2k$-gons $P_1^k,P_2^k,\ldots,P_r^k,$ of side length $1$. For each
$t\in[0,1]$, let $(X_{\mathrm{reg}})_t$ be the translation surface obtained
by replacing each $P_i^k$ with its deformation $(P_i^k)_t$ (up to isometry) as described in Proposition~\ref{even-gon-deform}.
Since opposite sides remain parallel throughout the deformation,
the prescribed side identifications continue to define a
translation surface. For example, Figure~\ref{fig:combined_all_hex_states deform}(B) shows $(X_{\mathrm{reg}})_t\in\mathcal{H}(1,1)$ at $t=0.5$, obtained by
gluing the equilateral hexagons $(P_1^3)_t$ and $(P_2^3)_t$,
deformed from $X_{\mathrm{reg}}\in\mathcal{P}_3\cap\mathcal{H}(1,1)$
shown in Figure~\ref{fig:combined_all_hex_states deform}(A).

By construction, the systole of $(X_{\mathrm{reg}})_t$ is
equal to $1$ for every $t\in[0,1]$. Moreover, by
Lemma~\ref{even-gon-monotonicity}, we have
\[
\operatorname{area}((P_i^k)_{t'})
>
\operatorname{area}((P_i^k)_t)
\qquad
\text{for all }t'<t\text{ in }[0,1]
\]
and for every $i=1,\ldots,r$. Summing over all the $2k$-gons,
we obtain
\[
\operatorname{area}((X_{\mathrm{reg}})_{t'})
>
\operatorname{area}((X_{\mathrm{reg}})_t)
\qquad
\text{for all }t'<t\text{ in }[0,1].
\]

We now normalize $(X_{\mathrm{reg}})_t$ and $(X_{\mathrm{reg}})_{t'}$
to unit area, denoting the resulting surfaces by
$(\widetilde{X}_{\mathrm{reg}})_t$ and
$(\widetilde{X}_{\mathrm{reg}})_{t'}$, respectively. Since lengths scale by the inverse square root of the area under this
normalization, the normalized systoles are
\[
\operatorname{sys}((\widetilde{X}_{\mathrm{reg}})_{t'})
=
\frac{1}{\sqrt{\operatorname{area}((X_{\mathrm{reg}})_{t'})}}
\]
and
\[
\operatorname{sys}((\widetilde{X}_{\mathrm{reg}})_{t})
=
\frac{1}{\sqrt{\operatorname{area}((X_{\mathrm{reg}})_t)}}.
\]
Consequently,
\[
\operatorname{sys}((\widetilde{X}_{\mathrm{reg}})_{t'})
<
\operatorname{sys}((\widetilde{X}_{\mathrm{reg}})_{t})
\quad
\text{for all }t'<t\text{ in }[0,1].
\]
For convenience, we continue to denote the unit-area
surfaces $(\widetilde{X}_{\mathrm{reg}})_{t}$ and
$(\widetilde{X}_{\mathrm{reg}})_{t'}$ by
$(X_{\mathrm{reg}})_t$ and $(X_{\mathrm{reg}})_{t'}$,
respectively.
Now, define
\[
\mathcal{F}:[0,1]\longrightarrow
\mathcal{H}(k_1,\ldots,k_n),
\quad
\mathcal{F}(t)=(X_{\mathrm{reg}})_t,
\]
where the surfaces are understood to be area-normalized.
Then $\mathcal{F}$ is a continuous deformation along which
the systole increases strictly. At $t=0$, we recover
$X_{\mathrm{reg}}$, while at $t=1$ we obtain a square-tiled
surface $O$, since each degenerate $2k$-gon is the boundary
of $k-1$ adjacent unit squares, with collinear vertices
retained as described in Proposition \ref{even-gon-deform}.

\begin{figure}[htbp]
\centering

\begin{subfigure}[b]{0.48\textwidth}
\centering
\begin{tikzpicture}[scale=1.1,line join=round,line cap=round]
\coordinate (A0) at (0:1);
\coordinate (A1) at (60:1);
\coordinate (A2) at (120:1);
\coordinate (A3) at (180:1);
\coordinate (A4) at (240:1);
\coordinate (A5) at (300:1);

\draw[thick] (A0)--(A1)--(A2)--(A3)--(A4)--(A5)--cycle;

\fill[red] (A0) circle (0.05);
\fill[blue] (A1) circle (0.05);
\fill[red] (A2) circle (0.05);
\fill[blue] (A3) circle (0.05);
\fill[red] (A4) circle (0.05);
\fill[blue] (A5) circle (0.05);

\node at ($(A0)!0.5!(A1)+(0.25,0)$) {$a$};
\node at ($(A1)!0.5!(A2)+(0,0.25)$) {$b$};
\node at ($(A2)!0.5!(A3)+(-0.25,0)$) {$c$};
\node at ($(A3)!0.5!(A4)+(-0.25,0)$) {$d$};
\node at ($(A4)!0.5!(A5)+(0,-0.25)$) {$e$};
\node at ($(A5)!0.5!(A0)+(0.25,0)$) {$f$};

\node[blue] at (0,0) {\tiny$P_1^3$};

\begin{scope}[xshift=2.5cm]
\coordinate (B0) at (0:1);
\coordinate (B1) at (60:1);
\coordinate (B2) at (120:1);
\coordinate (B3) at (180:1);
\coordinate (B4) at (240:1);
\coordinate (B5) at (300:1);

\draw[thick] (B0)--(B1)--(B2)--(B3)--(B4)--(B5)--cycle;

\fill[red] (B0) circle (0.05);
\fill[blue] (B1) circle (0.05);
\fill[red] (B2) circle (0.05);
\fill[blue] (B3) circle (0.05);
\fill[red] (B4) circle (0.05);
\fill[blue] (B5) circle (0.05);

\node at ($(B0)!0.5!(B1)+(0.25,0)$) {$d$};
\node at ($(B1)!0.5!(B2)+(0,0.25)$) {$e$};
\node at ($(B2)!0.5!(B3)+(-0.25,0)$) {$f$};
\node at ($(B3)!0.5!(B4)+(-0.25,0)$) {$a$};
\node at ($(B4)!0.5!(B5)+(0,-0.25)$) {$b$};
\node at ($(B5)!0.5!(B0)+(0.25,0)$) {$c$};

\node[blue] at (0,0) {\tiny$P_2^3$};
\end{scope}
\end{tikzpicture}
\caption{$X_{reg}\in\mathcal{P}_3\cap\mathcal{H}(1,1)$}
\label{fig:twohex_a}
\end{subfigure}
\hfill
\begin{subfigure}[b]{0.48\textwidth}
\centering
\begin{tikzpicture}[scale=1.1,line join=round,line cap=round]
\def\thOne{0}
\def\thTwo{30}
\def\thThree{105}

\coordinate (P0) at (0,0);
\path (P0) ++(\thOne:1) coordinate (P1);
\path (P1) ++(\thTwo:1) coordinate (P2);
\path (P2) ++(\thThree:1) coordinate (P3);
\path (P3) ++(180+\thOne:1) coordinate (P4);
\path (P4) ++(180+\thTwo:1) coordinate (P5);

\draw[thick] (P0)--(P1)--(P2)--(P3)--(P4)--(P5)--cycle;

\fill[red] (P0) circle (0.05);
\fill[blue] (P1) circle (0.05);
\fill[red] (P2) circle (0.05);
\fill[blue] (P3) circle (0.05);
\fill[red] (P4) circle (0.05);
\fill[blue] (P5) circle (0.05);

\node at ($(P0)!0.5!(P1)+(0,-0.25)$) {$e$};
\node at ($(P1)!0.5!(P2)+(0.2,-0.1)$) {$f$};
\node at ($(P2)!0.5!(P3)+(0.2,0.1)$) {$a$};
\node at ($(P3)!0.5!(P4)+(0,0.25)$) {$b$};
\node at ($(P4)!0.5!(P5)+(-0.2,0.1)$) {$c$};
\node at ($(P5)!0.5!(P0)+(-0.2,-0.1)$) {$d$};

\node[blue] at (0.8, 0.7) {\tiny$(P_1^3)_{0.5}$};

\begin{scope}[xshift=2.8cm]
\coordinate (Q0) at (0,0);
\path (Q0) ++(\thOne:1) coordinate (Q1);
\path (Q1) ++(\thTwo:1) coordinate (Q2);
\path (Q2) ++(\thThree:1) coordinate (Q3);
\path (Q3) ++(180+\thOne:1) coordinate (Q4);
\path (Q4) ++(180+\thTwo:1) coordinate (Q5);

\draw[thick] (Q0)--(Q1)--(Q2)--(Q3)--(Q4)--(Q5)--cycle;

\fill[red] (Q0) circle (0.05);
\fill[blue] (Q1) circle (0.05);
\fill[red] (Q2) circle (0.05);
\fill[blue] (Q3) circle (0.05);
\fill[red] (Q4) circle (0.05);
\fill[blue] (Q5) circle (0.05);

\node at ($(Q0)!0.5!(Q1)+(0,-0.25)$) {$b$};
\node at ($(Q1)!0.5!(Q2)+(0.2,-0.1)$) {$c$};
\node at ($(Q2)!0.5!(Q3)+(0.2,0.1)$) {$d$};
\node at ($(Q3)!0.5!(Q4)+(0,0.25)$) {$e$};
\node at ($(Q4)!0.5!(Q5)+(-0.2,0.1)$) {$f$};
\node at ($(Q5)!0.5!(Q0)+(-0.2,-0.1)$) {$a$};

\node[blue] at (0.8, 0.7) {\tiny$(P_2^3)_{0.5}$};
\end{scope}
\end{tikzpicture}
\caption{$(X_{reg})_t\in\mathcal{H}(1,1)$, $t=0.5$}
\label{fig:twohex_b}
\end{subfigure}

\vspace{0.8cm} 

\begin{subfigure}[b]{0.48\textwidth}
\centering
\begin{tikzpicture}[scale=1.1,line join=round,line cap=round]
\coordinate (P0) at (0,0);
\coordinate (P1) at (1,0);
\coordinate (P2) at (1,1);
\coordinate (P3) at (0,1);
\coordinate (P4) at (-1,1);
\coordinate (P5) at (-1,0);

\draw[thick] (P0)--(P1)--(P2)--(P3)--(P4)--(P5)--cycle;
\draw[densely dotted,thick] (P0)--(P3);

\fill[red] (P0) circle (0.05);
\fill[blue] (P1) circle (0.05);
\fill[red] (P2) circle (0.05);
\fill[blue] (P3) circle (0.05);
\fill[red] (P4) circle (0.05);
\fill[blue] (P5) circle (0.05);

\node at ($(P0)!0.5!(P1)+(0,-0.18)$) {$f$};
\node at ($(P1)!0.5!(P2)+(0.18,0)$) {$a$};
\node at ($(P2)!0.5!(P3)+(0,0.18)$) {$b$};
\node at ($(P3)!0.5!(P4)+(0,0.18)$) {$c$};
\node at ($(P4)!0.5!(P5)+(-0.18,0)$) {$d$};
\node at ($(P5)!0.5!(P0)+(0,-0.18)$) {$e$};

\node[blue, font=\tiny] at (0.5,0.5) {$(P_1^3)_1$};

\begin{scope}[xshift=3cm]
\coordinate (Q0) at (0,0);
\coordinate (Q1) at (1,0);
\coordinate (Q2) at (1,1);
\coordinate (Q3) at (0,1);
\coordinate (Q4) at (-1,1);
\coordinate (Q5) at (-1,0);

\draw[thick] (Q0)--(Q1)--(Q2)--(Q3)--(Q4)--(Q5)--cycle;
\draw[densely dotted,thick] (Q0)--(Q3);

\fill[red] (Q0) circle (0.05);
\fill[blue] (Q1) circle (0.05);
\fill[red] (Q2) circle (0.05);
\fill[blue] (Q3) circle (0.05);
\fill[red] (Q4) circle (0.05);
\fill[blue] (Q5) circle (0.05);

\node at ($(Q0)!0.5!(Q1)+(0,-0.18)$) {$c$};
\node at ($(Q1)!0.5!(Q2)+(0.18,0)$) {$d$};
\node at ($(Q2)!0.5!(Q3)+(0,0.18)$) {$e$};
\node at ($(Q3)!0.5!(Q4)+(0,0.18)$) {$f$};
\node at ($(Q4)!0.5!(Q5)+(-0.18,0)$) {$a$};
\node at ($(Q5)!0.5!(Q0)+(0,-0.18)$) {$b$};

\node[blue, font=\tiny] at (0.5,0.5) {$(P_2^3)_1$};
\end{scope}
\end{tikzpicture}
\caption{$(X_{reg})_1 = O$ in $\mathcal{H}(1,1)$}
\label{fig:sq-tiled}
\end{subfigure}
\hfill
\begin{subfigure}[b]{0.48\textwidth}
\centering
\begin{tikzpicture}[scale=1.1,line join=round,line cap=round]
\coordinate (P0) at (0,0);
\coordinate (P1) at (1,0);
\coordinate (P2) at (1.5,0.8660254);
\coordinate (P3) at (0.5,0.8660254);
\coordinate (P4) at (-0.5,0.8660254);
\coordinate (P5) at (-1,0);

\draw[thick] (P0)--(P1)--(P2)--(P3)--(P4)--(P5)--cycle;
\draw[densely dotted,thick] (P0)--(P3);
\draw[densely dotted,thick] (P0)--(P4);
\draw[densely dotted,thick] (P1)--(P3);

\fill[red] (P0) circle (0.05);
\fill[blue] (P1) circle (0.05);
\fill[red] (P2) circle (0.05);
\fill[blue] (P3) circle (0.05);
\fill[red] (P4) circle (0.05);
\fill[blue] (P5) circle (0.05);

\node at ($(P0)!0.5!(P1)+(0,-0.18)$) {$e$};
\node at ($(P1)!0.5!(P2)+(0.18,0.05)$) {$f$};
\node at ($(P2)!0.5!(P3)+(0,0.18)$) {$a$};
\node at ($(P3)!0.5!(P4)+(0,0.18)$) {$b$};
\node at ($(P4)!0.5!(P5)+(-0.18,0.05)$) {$c$};
\node at ($(P5)!0.5!(P0)+(0,-0.18)$) {$d$};

\begin{scope}[xshift=3.5cm]
\coordinate (Q0) at (0,0);
\coordinate (Q1) at (1,0);
\coordinate (Q2) at (1.5,0.8660254);
\coordinate (Q3) at (0.5,0.8660254);
\coordinate (Q4) at (-0.5,0.8660254);
\coordinate (Q5) at (-1,0);

\draw[thick] (Q0)--(Q1)--(Q2)--(Q3)--(Q4)--(Q5)--cycle;
\draw[densely dotted,thick] (Q0)--(Q3);
\draw[densely dotted,thick] (Q0)--(Q4);
\draw[densely dotted,thick] (Q1)--(Q3);

\fill[red] (Q0) circle (0.05);
\fill[blue] (Q1) circle (0.05);
\fill[red] (Q2) circle (0.05);
\fill[blue] (Q3) circle (0.05);
\fill[red] (Q4) circle (0.05);
\fill[blue] (Q5) circle (0.05);

\node at ($(Q0)!0.5!(Q1)+(0,-0.18)$) {$b$};
\node at ($(Q1)!0.5!(Q2)+(0.18,0.05)$) {$c$};
\node at ($(Q2)!0.5!(Q3)+(0,0.18)$) {$d$};
\node at ($(Q3)!0.5!(Q4)+(0,0.18)$) {$e$};
\node at ($(Q4)!0.5!(Q5)+(-0.18,0.05)$) {$f$};
\node at ($(Q5)!0.5!(Q0)+(0,-0.18)$) {$a$};
\end{scope}
\end{tikzpicture}
\caption{$\gamma(1) = X_{max}$ in $\mathcal{H}(1,1)$}
\label{fig:hex-max}
\end{subfigure}

\caption{Deformation of a surface $X_\mathrm{reg}\in\mathcal{P}_3\cap\mathcal{H}(1,1)$ to a maximal surface in $\mathcal{H}(1,1).$}
\label{fig:combined_all_hex_states deform}
\end{figure}

Now, by Case~$1$, there exists a continuous path
\[
\mathcal{D}:[0,1]\longrightarrow
\mathcal{H}(k_1,\ldots,k_n)
\]
such that
\[
\mathcal{D}(0)=O,\qquad
\mathcal{D}(1)=X_{\max},
\]
where $X_{\max}$ is a maximal surface and the systole
increases strictly along $\mathcal{D}$. For instance, see Figures~\ref{fig:combined_all_hex_states deform} (C) and (D).

Define
$\gamma:[0,1]\longrightarrow\mathcal{H}(k_1,\ldots,k_n)$
by
\[
\gamma(t)=
\begin{cases}
\mathcal{F}(2t), & 0\le t\le \frac12,\\[2mm]
\mathcal{D}(2t-1), & \frac12\le t\le 1.
\end{cases}
\]
Since $\mathcal{F}(1)=\mathcal{D}(0)=O$, the path $\gamma$
is continuous, with
\[
\gamma(0)=X_{\mathrm{reg}},
\qquad
\gamma(1)=X_{\max}.
\]
Furthermore, since the systole increases strictly along both
$\mathcal{F}$ and $\mathcal{D}$, we have
\[
\operatorname{sys}(\gamma(t'))
<
\operatorname{sys}(\gamma(t))
\qquad
\text{whenever }0\le t'<t\le1.
\]
Therefore, $\gamma$ is the required deformation from
$X_{\mathrm{reg}}$ to a maximal surface in
$\mathcal{H}(k_1,\ldots,k_n)$. This establishes Theorem~\ref{main result 1} for Case~$2$.

Finally, combining Cases~1 and 2, we conclude Theorem~\ref{main result 1}.

\end{proof}

We now extend the preceding construction to another class of translation surfaces. The following theorem (Theorem~\ref{thm:polygonal_deformation}) shows that these surfaces also admit a continuous systole-increasing deformation to maximal surfaces.

\begin{definition}[Forward and backward-pointing edges]
Let $P$ be a regular $2m$-gon with its boundary equipped with the standard counterclockwise orientation. The $2m$ edges of $P$ can be viewed as vectors $v_k \in \mathbb{R}^2$ for $k = 1, \dots, 2m$. Assuming $P$ is oriented such that $v_1$ is horizontal and points in the positive real direction, the edges naturally partition into two sets of cardinality $m$:
\begin{enumerate}
    \item \textbf{Forward-pointing edges:} The $m$ consecutive edges $v_i$ for $i \in \{1, \dots, m\}$, which form an angle of $\theta_i = \frac{i - 1}{m}\pi \in [0, \pi)$ with the positive real axis.
    \item \textbf{Backward-pointing edges:} The remaining $m$ edges $v_{i+m}$ for $i \in \{1, \dots, m\}$, which form an angle of $\theta_{i+m} = \frac{i - 1}{m}\pi + \pi \in [\pi, 2\pi)$ with the positive real axis.
\end{enumerate}
For any backward-pointing edge $v_{i+m}$, its \emph{forward opposite} is the unique forward-pointing edge $v_i$ that is geometrically parallel but oppositely oriented along the boundary. For example, see Figure~\ref{fig:octagon_vectors}.
\end{definition}

\begin{figure}[htbp]
    \centering
    \begin{tikzpicture}[
        scale=1.2,
        forward/.style={thick, blue, postaction={decorate}, decoration={markings, mark=at position 0.55 with {\arrow{>}}}},
        backward/.style={thick, red, postaction={decorate}, decoration={markings, mark=at position 0.55 with {\arrow{>}}}}
    ]

        \coordinate (V0) at (0,0);

        \draw[forward] (V0) -- ++(0:2) coordinate (V1) 
            node[midway, below=3pt, black] {$v_1$ ($0$)};
            
        \draw[forward] (V1) -- ++(45:2) coordinate (V2) 
            node[midway, right=3pt, black] {$v_2$ ($\frac{\pi}{4}$)};
            
        \draw[forward] (V2) -- ++(90:2) coordinate (V3) 
            node[midway, right=3pt, black] {$v_3$ ($\frac{\pi}{2}$)};
            
        \draw[forward] (V3) -- ++(135:2) coordinate (V4) 
            node[midway, above right=3pt, black] {$v_4$ ($\frac{3\pi}{4}$)};

        \draw[backward] (V4) -- ++(180:2) coordinate (V5) 
            node[midway, above=3pt, black] {$v_5$ ($\pi$)};
            
        \draw[backward] (V5) -- ++(225:2) coordinate (V6) 
            node[midway, left=3pt, black] {$v_6$ ($\frac{5\pi}{4}$)};
            
        \draw[backward] (V6) -- ++(270:2) coordinate (V7) 
            node[midway, left=3pt, black] {$v_7$ ($\frac{3\pi}{2}$)};
            
        \draw[backward] (V7) -- ++(315:2) coordinate (V8) 
            node[midway, below left=3pt, black] {$v_8$ ($\frac{7\pi}{4}$)};

    \end{tikzpicture}
    \caption{A regular octagon ($2m=8$) with its boundary edges represented as Euclidean vectors. The forward-pointing edges (blue) map to angles in $[0, \pi)$, while the backward-pointing edges (red) map to angles in $[\pi, 2\pi)$. By construction, every backward-pointing edge $v_{i+4}$ is the additive inverse of its forward opposite $v_i$.}
    \label{fig:octagon_vectors}
\end{figure}

\begin{definition}[Coherent alignment]\label{def:coherent_alignment}
Let $X$ be a translation surface formed by gluing a collection of regular $2m_j$-gons ($m_j \ge 3$). Consider an identification between a forward-pointing edge with index $i_1 \in \{1, \dots, m_1\}$ on a $2m_1$-gon, and a backward-pointing edge whose forward opposite has index $i_2 \in \{1, \dots, m_2\}$ on a $2m_2$-gon. This edge pairing is said to be \emph{coherently aligned} if
\[
\frac{i_1 - 1}{m_1} = \frac{i_2 - 1}{m_2}.
\]
A polygonal decomposition of $X$ is \emph{coherently aligned} if every pair of identified edges in the surface satisfies this condition. For instance, see Figure~\ref{fig:coherent_surface_octagon_hexagon}.
\end{definition}

\begin{figure}[htbp]
    \centering
    \begin{tikzpicture}[
        scale=1.4,
        forward/.style={thick, blue, postaction={decorate}, decoration={markings, mark=at position 0.55 with {\arrow{>}}}},
        backward/.style={thick, red, postaction={decorate}, decoration={markings, mark=at position 0.55 with {\arrow{>}}}},
        cross/.style={thick, violet, postaction={decorate}, decoration={markings, mark=at position 0.55 with {\arrow{>}}}},
        neutral/.style={draw=gray!50}
    ]

        \begin{scope}[shift={(0,0)}]
            
            \draw[cross] (0,0) -- ++(0:1) node[midway, below] {$e_{1,1}$};
            
            \draw[forward] (1,0) -- ++(45:1) node[midway, right] {$e_{1,2}$};
            \draw[forward] (1+0.707, 0.707) -- ++(90:1) node[midway, right] {$e_{1,3}$};
            \draw[forward] (1+0.707, 1.707) -- ++(135:1) node[midway, above right] {$e_{1,4}$};
            
            \draw[cross] (1, 2.414) -- ++(180:1) node[midway, above] {$e'_{1,1}$};
            
            \draw[backward] (0, 2.414) -- ++(225:1) node[midway, left] {$e'_{1,4}$};
            \draw[backward] (-0.707, 1.707) -- ++(270:1) node[midway, left] {$e'_{1,3}$};
            \draw[backward] (-0.707, 0.707) -- ++(315:1) node[midway, below left] {$e'_{1,2}$};
        \end{scope}

        \begin{scope}[shift={(3.8, 0.35)}]
            
            \draw[cross] (0,0) -- ++(0:1) node[midway, below] {$e_{2,1}$};
            
            \draw[forward] (1,0) -- ++(60:1) node[midway, right] {$e_{2,2}$};
            \draw[forward] (1.5, 0.866) -- ++(120:1) node[midway, above right] {$e_{2,3}$};
            
            \draw[cross] (1, 1.732) -- ++(180:1) node[midway, above] {$e'_{2,1}$};
            
            \draw[backward] (0, 1.732) -- ++(240:1) node[midway, left] {$e'_{2,3}$};
            \draw[backward] (-0.5, 0.866) -- ++(300:1) node[midway, below left] {$e'_{2,2}$};
            
        \end{scope}

    \end{tikzpicture}
    \caption{A coherently aligned translation surface $X$ formed by one regular octagon and one regular hexagon. The cross-polygon gluings (violet) connect $e_{1,1} \sim e'_{2,1}$ and $e_{2,1} \sim e'_{1,1}$, while all non-horizontal edges (blue and red) self-glue to their respective opposite edges.}
    \label{fig:coherent_surface_octagon_hexagon}
\end{figure}

\begin{proposition}\label{prop:coherent_parallel}
Let $X$ be a translation surface decomposed into a coherently aligned collection of regular unit $2m_j$-gons $j = 1, \dots, r$, with $m_j \ge 3$. Then, under the deformations $\{P_t^{m_j}\}_{t \in [0,1]}$ from Proposition~\ref{even-gon-deform}, every pair of glued edges remains perfectly parallel, equal in length, and oppositely oriented, yielding a valid translation surface $X_t$ for all $t \in [0,1]$.
\end{proposition}

\begin{proof}
Let $e_a$ and $e_b$ be two arbitrary edges in some $2m_a$-gon and $2m_b$-gon respectively which are identified in the initial translation surface $X.$ Because translation surface gluings identify forward-pointing vectors with backward-pointing vectors, we may assume $e_a$ corresponds to a forward-pointing index $i_a \in \{1, \dots, m_a\}$, and $e_b$ corresponds to a backward-pointing index $i_b + m_b$, where $i_b \in \{1, \dots, m_b\}$ is its forward-pointing opposite.

By Proposition~\ref{even-gon-deform}, when these polygons are treated as isolated, unembedded shapes in their standard positions, the local direction angles of these edges under the deformation $P_t^k$ are given by:
\begin{align*}
\theta_{i_a}(t) &= (i_a - 1)\frac{(1-t)\pi}{m_a}, \\
\theta_{i_b + m_b}(t) &= (i_b - 1)\frac{(1-t)\pi}{m_b} + \pi.
\end{align*}
Because the abstract gluing defining $X$ does not \textit{a priori} constrain the global orientation of the polygons, they may differ from their standard horizontal positions. To rigorously account for these arbitrary initial embedding rotations, we introduce $\alpha_a$ and $\alpha_b$. The true geometric direction angles of the edges $e_a$ and $e_b$ within the surface at any time $t \in [0,1]$ are therefore:
\begin{align}
\Theta_a(t) &= \alpha_a + (i_a - 1)\frac{(1-t)\pi}{m_a}, \label{eq:theta_a} \\
\Theta_b(t) &= \alpha_b + (i_b - 1)\frac{(1-t)\pi}{m_b} + \pi. \label{eq:theta_b}
\end{align}

At $t=0$, the surface $X_0 = X$ is a valid translation surface. Therefore, the identified edges $e_a$ and $e_b$ must be strictly parallel and oppositely oriented, which dictates that the difference in their true direction angles must be exactly $\pi$, meaning $\Theta_b(0) - \Theta_a(0) = \pi$. Evaluating our direction angle equations at $t=0$ yields:
\[
\alpha_b + (i_b - 1)\frac{\pi}{m_b} + \pi - \left( \alpha_a + (i_a - 1)\frac{\pi}{m_a} \right) = \pi,
\]
which simplifies to:
\begin{equation}\label{eq:alpha_diff}
\alpha_a - \alpha_b = (i_b - 1)\frac{\pi}{m_b} - (i_a - 1)\frac{\pi}{m_a}.
\end{equation}

By the hypothesis of coherent alignment, we have equality of the angular velocity fractions. Let $c$ be this shared constant:
\[
c = \frac{i_a - 1}{m_a} = \frac{i_b - 1}{m_b}.
\]
Substituting this into Equation~\ref{eq:alpha_diff} gives $\alpha_a - \alpha_b = c\pi - c\pi = 0$, implying $\alpha_a = \alpha_b$. This establishes a rigid geometric consequence: coherently aligned polygons must share identical initial embedding rotations within the surface.

To verify that the edges remain properly aligned for all $t>0$, we evaluate the difference in their geometric direction angles $\Theta_b(t) - \Theta_a(t)$ at an arbitrary time $t \in [0,1]$. Utilizing $\alpha_a = \alpha_b$ and the coherent alignment constant $c$, Equations~\ref{eq:theta_a} and \ref{eq:theta_b} yield:
\begin{align*}
\Theta_b(t) - \Theta_a(t) &= \left( \alpha_b + c(1-t)\pi + \pi \right) - \left( \alpha_a + c(1-t)\pi \right) \\
&= (\alpha_b - \alpha_a) + \pi + c(1-t)\pi - c(1-t)\pi \\
&= \pi.
\end{align*}
Thus, the difference between the direction angles of the two edges remains exactly $\pi$ for all $t \in [0,1]$. Furthermore, Proposition~\ref{even-gon-deform} establishes that the internal deformation preserves unit side lengths identically for all $t$. Because $e_a$ and $e_b$ continuously maintain equal length, parallel alignment, and opposite orientation, the prescribed gluing map remains a valid Euclidean isometry. Since this holds for every pair of identified edges in the coherently aligned decomposition, the global surface $X_t$ remains a valid translation surface.
\end{proof}

\begin{theorem}\label{thm:polygonal_deformation}
Let $X \in \mathcal{H}(k_1,\ldots,k_n)$ be a translation surface formed by gluing together various regular $2k$-gons ($k \ge 3$). Suppose that every polygon has a side length equal to $\operatorname{sys}(X)$, and that the gluing is coherently aligned. Then there exists a continuous map
\[
\gamma:[0,1]\longrightarrow\mathcal{H}(k_1,\ldots,k_n)
\]
such that
\begin{enumerate}
    \item $\gamma(0)=X$ and $\gamma(1)$ is a maximal surface;
    \item $\operatorname{sys}(\gamma(t'))<\operatorname{sys}(\gamma(t))$ for all $t'<t$ in $[0,1]$.
\end{enumerate}
\end{theorem}

\begin{proof}
By hypothesis, $X$ is constructed from a coherently aligned collection of regular $2m_j$-gons ($m_j \ge 3$), which we denote by $P_1^{m_1}, \ldots, P_r^{m_r}$. Without loss of generality, we globally scale the initial surface $X$ such that $\operatorname{sys}(X)=1$, ensuring all constituent polygons have unit side length.

For each $t\in[0,1]$, we independently deform each polygon $P_j^{m_j}$ into $(P_j^{m_j})_t$ using the continuous family of mappings described in Proposition~\ref{even-gon-deform}. It is precisely here that the coherent alignment condition is required. Because the original gluing is coherently aligned, Proposition~\ref{prop:coherent_parallel} guarantees that every pair of identified edges  remains strictly parallel, oppositely oriented, and of unit length for all $t \in [0,1]$. This completely prevents the surface from tearing, ensuring the initial gluing combinatorics remain valid Euclidean isometries. We may thus glue the deformed polygons together to form a continuous family of unnormalized translation surfaces, denoted $\{X_t\}_{t\in[0,1]}$, which trivially maintains a constant systole $\operatorname{sys}(X_t) = 1$.

By Lemma~\ref{even-gon-monotonicity}, because $m_j \ge 3$ for every polygon, the area of every deforming $2m_j$-gon in the decomposition decreases strictly as $t$ increases:
\[
\operatorname{area}((P_j^{m_j})_{t'}) > \operatorname{area}((P_j^{m_j})_t) \qquad \text{for all } t'<t \text{ in } [0,1].
\]

\begin{figure}[htpb]
    \centering
    \begin{tikzpicture}[scale=1.4, thick]
        \node at (4.5, 3.5) {\Large \text{Stage 1: $t=0$ (Initial Surface $X$)}};

        \begin{scope}[shift={(0,0)}]
            \draw[fill=blue!10] 
                (0,0) coordinate (H1_1)
                -- ++(0:1) coordinate (H1_2) node[midway, below] {\tiny $e_1^1$}
                -- ++(60:1) coordinate (H1_3) node[midway, right] {\tiny $e_2^1$}
                -- ++(120:1) coordinate (H1_4) node[midway, above right] {\tiny $e_3^1$}
                -- ++(180:1) coordinate (H1_5) node[midway, above] {\tiny $e_4^1$}
                -- ++(240:1) coordinate (H1_6) node[midway, left] {\tiny $e_5^1$}
                -- cycle node[midway, below left] {\tiny $e_6^1$};
            \node at (0.5, 0.866) {$P_1^3$};
            \foreach \i in {1,...,6} \fill[red] (H1_\i) circle (1pt);
        \end{scope}

        \begin{scope}[shift={(3.5, -0.5)}]
            \draw[fill=red!10] 
                (0,0) coordinate (O_1)
                -- ++(0:1) coordinate (O_2) node[midway, below] {\tiny $e_1^2$}
                -- ++(45:1) coordinate (O_3) node[midway, below right] {\tiny $e_2^2$}
                -- ++(90:1) coordinate (O_4) node[midway, right] {\tiny $e_3^2$}
                -- ++(135:1) coordinate (O_5) node[midway, above right] {\tiny $e_4^2$}
                -- ++(180:1) coordinate (O_6) node[midway, above] {\tiny $e_5^2$}
                -- ++(225:1) coordinate (O_7) node[midway, above left] {\tiny $e_6^2$}
                -- ++(270:1) coordinate (O_8) node[midway, left] {\tiny $e_7^2$}
                -- cycle node[midway, below left] {\tiny $e_8^2$};
            \node at (0.5, 1.2) {$P_2^4$};
            \foreach \i in {1,...,8} \fill[red] (O_\i) circle (1pt);
        \end{scope}

        \begin{scope}[shift={(7.5, 0)}]
            \draw[fill=green!10] 
                (0,0) coordinate (H2_1)
                -- ++(0:1) coordinate (H2_2) node[midway, below] {\tiny $e_1^3$}
                -- ++(60:1) coordinate (H2_3) node[midway, right] {\tiny $e_2^3$}
                -- ++(120:1) coordinate (H2_4) node[midway, above right] {\tiny $e_3^3$}
                -- ++(180:1) coordinate (H2_5) node[midway, above] {\tiny $e_4^3$}
                -- ++(240:1) coordinate (H2_6) node[midway, left] {\tiny $e_5^3$}
                -- cycle node[midway, below left] {\tiny $e_6^3$};
            \node at (0.5, 0.866) {$P_3^3$};
            \foreach \i in {1,...,6} \fill[red] (H2_\i) circle (1pt);
        \end{scope}
    \end{tikzpicture}

    \vspace{1cm}

    \begin{tikzpicture}[scale=1.4, thick]
        \node at (4.5, 2.5) {\Large \text{Stage 2: $0 < t < 1$ (Intermediate $X_t$)}};

        \begin{scope}[shift={(0.5,0)}]
            \draw[fill=blue!10] 
                (0,0) coordinate (M1_1)
                -- ++(0:1) coordinate (M1_2) node[midway, below] {\tiny $e_1^1$}
                -- ++(30:1) coordinate (M1_3) node[midway, below right] {\tiny $e_2^1$}
                -- ++(60:1) coordinate (M1_4) node[midway, right] {\tiny $e_3^1$}
                -- ++(180:1) coordinate (M1_5) node[midway, above] {\tiny $e_4^1$}
                -- ++(210:1) coordinate (M1_6) node[midway, above left] {\tiny $e_5^1$}
                -- cycle node[midway, left] {\tiny $e_6^1$};
            \node at (1.2, 0.8) {$(P_1^3)_t$};
            \foreach \i in {1,...,6} \fill[red] (M1_\i) circle (1pt);
        \end{scope}

        \begin{scope}[shift={(3.5, 0)}]
            \draw[fill=red!10] 
                (0,0) coordinate (M2_1)
                -- ++(0:1) coordinate (M2_2) node[midway, below] {\tiny $e_1^2$}
                -- ++(22.5:1) coordinate (M2_3) node[midway, below right] {\tiny $e_2^2$}
                -- ++(45:1) coordinate (M2_4) node[midway, right] {\tiny $e_3^2$}
                -- ++(67.5:1) coordinate (M2_5) node[midway, above right] {\tiny $e_4^2$}
                -- ++(180:1) coordinate (M2_6) node[midway, above] {\tiny $e_5^2$}
                -- ++(202.5:1) coordinate (M2_7) node[midway, above left] {\tiny $e_6^2$}
                -- ++(225:1) coordinate (M2_8) node[midway, left] {\tiny $e_7^2$}
                -- cycle node[midway, above left] {\tiny $e_8^2$};
            \node at (1.2, 0.8) {$(P_2^4)_t$};
            \foreach \i in {1,...,8} \fill[red] (M2_\i) circle (1pt);
        \end{scope}

        \begin{scope}[shift={(7, 0)}]
            \draw[fill=green!10] 
                (0,0) coordinate (M3_1)
                -- ++(0:1) coordinate (M3_2) node[midway, below] {\tiny $e_1^3$}
                -- ++(30:1) coordinate (M3_3) node[midway, below right] {\tiny $e_2^3$}
                -- ++(60:1) coordinate (M3_4) node[midway, right] {\tiny $e_3^3$}
                -- ++(180:1) coordinate (M3_5) node[midway, above] {\tiny $e_4^3$}
                -- ++(210:1) coordinate (M3_6) node[midway, above left] {\tiny $e_5^3$}
                -- cycle node[midway, left] {\tiny $e_6^3$};
            \node at (1.2, 0.7) {$(P_3^3)_t$};
            \foreach \i in {1,...,6} \fill[red] (M3_\i) circle (1pt);
        \end{scope}
    \end{tikzpicture}

    \vspace{1cm}

    \begin{tikzpicture}[scale=1.4, thick]
        \node at (4.5, 2.5) {\Large \text{Stage 3: $t=1$ (Square-Tiled Limit $O$)}};

        \begin{scope}[shift={(0.5,0)}]
            \draw[fill=blue!10] 
                (0,0) -- (1,0) node[midway, below] {\tiny $e_1^1$}
                -- (2,0) node[midway, below] {\tiny $e_2^1$}
                -- (2,1) node[midway, right] {\tiny $e_3^1$}
                -- (1,1) node[midway, above] {\tiny $e_4^1$}
                -- (0,1) node[midway, above] {\tiny $e_5^1$}
                -- cycle node[midway, left] {\tiny $e_6^1$};
            \draw[dashed] (1,0) -- (1,1);
            \node at (0.5, 0.5) {$(P_1^3)_1$};
            \foreach \x in {0,1,2} { \foreach \y in {0,1} { \fill[red] (\x,\y) circle (1pt); } }
        \end{scope}

        \begin{scope}[shift={(3.5, 0)}]
            \draw[fill=red!10] 
                (0,0) -- (1,0) node[midway, below] {\tiny $e_1^2$}
                -- (2,0) node[midway, below] {\tiny $e_2^2$}
                -- (3,0) node[midway, below] {\tiny $e_3^2$}
                -- (3,1) node[midway, right] {\tiny $e_4^2$}
                -- (2,1) node[midway, above] {\tiny $e_5^2$}
                -- (1,1) node[midway, above] {\tiny $e_6^2$}
                -- (0,1) node[midway, above] {\tiny $e_7^2$}
                -- cycle node[midway, left] {\tiny $e_8^2$};
            \draw[dashed] (1,0) -- (1,1);
            \draw[dashed] (2,0) -- (2,1);
            \node at (1.5, 0.5) {$(P_2^4)_1$};
            \foreach \x in {0,1,2,3} { \foreach \y in {0,1} { \fill[red] (\x,\y) circle (1pt); } }
        \end{scope}

        \begin{scope}[shift={(7.5, 0)}]
            \draw[fill=green!10] 
                (0,0) -- (1,0) node[midway, below] {\tiny $e_1^3$}
                -- (2,0) node[midway, below] {\tiny $e_2^3$}
                -- (2,1) node[midway, right] {\tiny $e_3^3$}
                -- (1,1) node[midway, above] {\tiny $e_4^3$}
                -- (0,1) node[midway, above] {\tiny $e_5^3$}
                -- cycle node[midway, left] {\tiny $e_6^3$};
            \draw[dashed] (1,0) -- (1,1);
            \node at (0.5, 0.5) {$(P_3^3)_1$};
            \foreach \x in {0,1,2} { \foreach \y in {0,1} { \fill[red] (\x,\y) circle (1pt); } }
        \end{scope}
        \end{tikzpicture}
    
    \caption{States of deformation of the surface $X$}
    \label{fig: new class}
\end{figure}

\begin{figure}[htpb]
    \centering
    \begin{tikzpicture}[scale=1.5, thick]
        \begin{scope}[shift={(0,0)}]
            \draw[fill=blue!10] 
                (0,0) -- (1,0) node[midway, below] {\tiny $e_1^1$}
                -- (2,0) node[midway, below] {\tiny $e_2^1$}
                -- (2.5, 0.866) node[midway, right] {\tiny $e_3^1$}
                -- (1.5, 0.866) node[midway, above] {\tiny $e_4^1$}
                -- (0.5, 0.866) node[midway, above] {\tiny $e_5^1$}
                -- cycle node[midway, left] {\tiny $e_6^1$};
            \draw[dashed] (1,0) -- (1.5, 0.866);
            \draw[dashed] (1,0) -- (0.5, 0.866);
            \draw[dashed] (2,0) -- (1.5, 0.866);
           
            \foreach \x/\y in {0/0, 1/0, 2/0, 0.5/0.866, 1.5/0.866, 2.5/0.866} {
                \fill[red] (\x,\y) circle (1pt);
            }
        \end{scope}

        \begin{scope}[shift={(3.2, 0)}]
            \draw[fill=red!10] 
                (0,0) -- (1,0) node[midway, below] {\tiny $e_1^2$}
                -- (2,0) node[midway, below] {\tiny $e_2^2$}
                -- (3,0) node[midway, below] {\tiny $e_3^2$}
                -- (3.5, 0.866) node[midway, right] {\tiny $e_4^2$}
                -- (2.5, 0.866) node[midway, above] {\tiny $e_5^2$}
                -- (1.5, 0.866) node[midway, above] {\tiny $e_6^2$}
                -- (0.5, 0.866) node[midway, above] {\tiny $e_7^2$}
                -- cycle node[midway, left] {\tiny $e_8^2$};
            \draw[dashed] (1,0) -- (1.5, 0.866);
            \draw[dashed] (2,0) -- (2.5, 0.866);
            \draw[dashed] (1,0) -- (0.5, 0.866);
            \draw[dashed] (2,0) -- (1.5, 0.866);
            \draw[dashed] (3,0) -- (2.5, 0.866);
         
            \foreach \x/\y in {0/0, 1/0, 2/0, 3/0, 0.5/0.866, 1.5/0.866, 2.5/0.866, 3.5/0.866} {
                \fill[red] (\x,\y) circle (1pt);
            }
        \end{scope}

        \begin{scope}[shift={(7.3, 0)}]
            \draw[fill=green!10] 
                (0,0) -- (1,0) node[midway, below] {\tiny $e_1^3$}
                -- (2,0) node[midway, below] {\tiny $e_2^3$}
                -- (2.5, 0.866) node[midway, right] {\tiny $e_3^3$}
                -- (1.5, 0.866) node[midway, above] {\tiny $e_4^3$}
                -- (0.5, 0.866) node[midway, above] {\tiny $e_5^3$}
                -- cycle node[midway, left] {\tiny $e_6^3$};
            \draw[dashed] (1,0) -- (1.5, 0.866);
            \draw[dashed] (1,0) -- (0.5, 0.866);
            \draw[dashed] (2,0) -- (1.5, 0.866);
          
            \foreach \x/\y in {0/0, 1/0, 2/0, 0.5/0.866, 1.5/0.866, 2.5/0.866} {
                \fill[red] (\x,\y) circle (1pt);
            }
        \end{scope}
    \end{tikzpicture}
    \caption{Maximal surface deformed from $X$}
    \label{fig:maximal_surface_limit}
\end{figure}
Since the total area of $X_t$ is simply the sum of the areas of its constituent polygons, the macroscopic area of the unnormalized surface strictly decreases as well:
\[
\operatorname{area}(X_{t'}) > \operatorname{area}(X_t) \qquad \text{for all } t' < t.
\]
We now normalize $X_t$ and $X_{t'}$ to unit area, denoting the resulting surfaces by $\tilde{X}_t$ and $\tilde{X}_{t'}$, respectively. Because length scales by the inverse square root of area under this normalization, the normalized systoles are given by
\[
\operatorname{sys}(\tilde{X}_{t'}) = \frac{1}{\sqrt{\operatorname{area}(X_{t'})}}
\quad \text{and} \quad
\operatorname{sys}(\tilde{X}_t) = \frac{1}{\sqrt{\operatorname{area}(X_t)}}.
\]
Since $\operatorname{area}(X_{t'}) > \operatorname{area}(X_t)$, it follows immediately that
\[
\operatorname{sys}(\tilde{X_{t'}}) < \operatorname{sys}(\tilde{X_t}) \quad \text{for all } t'<t \text{ in } [0,1].
\]
For convenience, we continue to denote the strictly systole-increasing, unit-area surfaces $\tilde{X_t}$ simply by $X_t$. 

Now, define the continuous deformation path $\mathcal{G}:[0,1] \longrightarrow \mathcal{H}(k_1,\ldots,k_n)$ by $\mathcal{G}(t) = X_t$. At $t=0$, we recover the initial surface $X$. At $t=1$, Proposition~\ref{even-gon-deform} establishes that every degenerate $2m_j$-gon collapses perfectly into the boundaries of adjacent unit squares. Consequently, $\mathcal{G}(1) = O$ is a square-tiled surface.

By Case~$1$ in Theorem~\ref{main result 1}, there is a continuous path $\mathcal{E}:[0,1] \longrightarrow \mathcal{H}(k_1,\ldots,k_n)$
such that $\mathcal{E}(0) = O$ and $\mathcal{E}(1) = X_{\max}$, where $X_{\max}$ is a maximal surface decomposed by equilateral triangles. Furthermore, the systole increases strictly along $\mathcal{E}$.

Finally, we concatenate these paths to define $\gamma:[0,1] \longrightarrow \mathcal{H}(k_1,\ldots,k_n)$ by
\[
\gamma(t) = 
\begin{cases}
\mathcal{G}(2t), & 0 \le t \le \frac{1}{2}, \\[2mm]
\mathcal{E}(2t-1), & \frac{1}{2} \le t \le 1.
\end{cases}
\]
Since $\mathcal{G}(1) = \mathcal{E}(0) = O$, the path $\gamma$ is continuous. It satisfies condition (1) because $\gamma(0) = \mathcal{G}(0) = X$ and $\gamma(1) = \mathcal{E}(1) = X_{\max}$. It satisfies condition (2) because the systole increases strictly along both constituent paths $\mathcal{G}$ and $\mathcal{E}$, guaranteeing that
\[
\operatorname{sys}(\gamma(t')) < \operatorname{sys}(\gamma(t)) \qquad \text{for all } t' < t \text{ in } [0,1].
\]
For instance, Figure~\ref{fig: new class} illustrates all stages of the continuous deformation for a translation surface defined by the following edge identifications: $e_1^1\sim e_5^2$, $e_1^2\sim e_4^3$, $e_2^1\sim e_5^3$, $e_3^1\sim e_6^3$, $e_4^1\sim e_1^3$, $e_5^1\sim e_2^3$, $e_6^1\sim e_3^3$, $e_2^2\sim e_6^2$, $e_3^2\sim e_7^2$, and $e_4^2\sim e_8^2$. Figure~\ref{fig:maximal_surface_limit} depicts the the resulting maximal surface.

This completes the proof.
\end{proof}


\section*{Declaration of generative AI and AI-assisted technologies in the manuscript preparation process}

During the preparation of this work, the authors used OpenAI’s ChatGPT to improve the clarity and readability of portions of the manuscript. After using this tool, the authors reviewed and edited the content as needed, independently verified the mathematical content, arguments, and references, and take full responsibility for the accuracy, integrity, and originality of the manuscript and for the final version submitted for publication.

\bibliographystyle{plain}
\bibliography{bibliography}

\end{document}